\documentclass[12pt,twoside]{amsart}
\usepackage{amssymb}
\usepackage{verbatim}
\usepackage{amsmath}
\usepackage{bm}
\usepackage{a4wide}
\usepackage[T1]{fontenc}
\usepackage{times}
\usepackage{amssymb,latexsym}
\usepackage{enumerate}
\usepackage{pict2e}

\makeatletter
\DeclareRobustCommand{\intprod}{%
  \mathbin{\mathpalette\int@prod{(0.1,0)(0.9,0)(0.9,0.8)}}%
}
\DeclareRobustCommand{\intprodr}{%
  \mathbin{\mathpalette\int@prod{(0.1,0.8)(0.1,0)(0.9,0)}}}

\newcommand{\int@prod}[2]{%
  \begingroup
  \sbox\z@{$\m@th#1+$}%
  \setlength\unitlength{\wd\z@}%
  \begin{picture}(1,1)
  \roundcap
  \polyline#2
  \end{picture}%
  \endgroup
}
\makeatother

\makeatletter
\newcommand{\sumprime}{\if@display\sideset{}{'}\sum%
            \else\sum'\fi}
\makeatother

\allowdisplaybreaks
\begin{document}

\numberwithin{equation}{section}

\newtheorem{theorem}{Theorem}[section]
\newtheorem{proposition}[theorem]{Proposition}
\newtheorem{conjecture}[theorem]{Conjecture}
\def\theconjecture{\unskip}
\newtheorem{corollary}[theorem]{Corollary}
\newtheorem{lemma}[theorem]{Lemma}
\newtheorem{observation}[theorem]{Observation}
\newtheorem{definition}{Definition}
\newtheorem*{definition*}{Definition}
\numberwithin{definition}{section} 
\newtheorem{remark}{Remark}
\newtheorem*{note}{Note}
\def\theremark{\unskip}
\newtheorem{kl}{Key Lemma}
\def\thekl{\unskip}
\newtheorem{question}{Question}
\def\thequestion{\unskip}
\newtheorem*{example}{Example}
\newtheorem{problem}{Problem}

\thanks{}

\title{Bergman functions on weakly uniformly perfect domains II}

 \author[Zhiyuan Zheng]{Zhiyuan Zheng}
\date{2026. 07. 16}



\address[Zhiyuan Zheng]{Department of Mathematics and Computer Sciences, Tongling University, Anhui, 244000, China}

\email{2023052@tlu.edu.cn}

\begin{abstract}

We study the boundary asymptotic behavior of Bergman functions on planar domains. Motivated by Chen's question on the equivalence between uniform perfectness of the boundary and the sharp growth rates of the Bergman kernel and the Bergman metric, we focus on the second part of the question concerning the Bergman metric. We prove that $\partial\Omega$ is uniformly perfect if and only if $K_{\Omega}^{(1)}(w)\asymp \delta_{\Omega}(w)^{-4}$. We also find that under suitable weak uniform perfectness conditions, there exist sequences of points along which $b_{\Omega}(w_n)=o(\delta_{\Omega}(w_n)^{-1})$, providing partial evidence toward an affirmative answer. Our method relies on sharp lower and upper bounds for $K_{\Omega}^{(1)}$ and $K_{\Omega}$. As an application, we obtain corresponding lower bounds for the Bergman distance on certain planar domains.



\end{abstract}

\maketitle

\section{Introduction}

The boundary asymptotic behavior of Bergman functions is an important problem in complex analysis, and there is a substantial body of literature on this topic (see, e.g., \cite{Fefferman, BlockiPflug1998, Chen1999, Herbort1999, JarnickiPflug1989, JarnickiPflugZwonek, Ohsawa1984, Ohsawa1993, PflugZwonek2005, Zwonek1999}). In this paper, we are concerned with Bergman functions on certain planar domains. Let $A^2(\Omega)$ denote the Bergman space on a domain $\Omega \subset \mathbb{C}$. For $z, w \in \Omega$, let $K_{\Omega}(z,w)$ be the Bergman kernel function of $\Omega$, write $K_{\Omega}(z):=K_{\Omega}(z,z)$, and let $b_{\Omega}(z)^2 \mathrm{d}z \otimes \mathrm{d}\overline{z}$ denote the Bergman metric on $\Omega$.

We denote
$$
K^{(1)}_{\Omega}(z):= \sup \{  |f'(z)|^2: f\in A^2(\Omega), f(z)=0, \|f\|_{L^2(\Omega)}=1      \}.
$$
Then
$$
b_{\Omega}(z)^2=\frac{K^{(1)}_{\Omega}(z)}{K_{\Omega}(z)}.
$$
For two points $z_0,z \in \Omega$, we denote their Bergman distance by $d_{\Omega}(z_0,z)$.

In 2013, Chen \cite{Chen2013} linked the boundary behavior of Bergman functions on planar domains to the classical notion of uniform perfectness. Here, uniform perfectness is a significant concept in complex analysis which connects many branches of mathematics, including potential theory, fractal geometry, dynamical systems, and spectral theory (see, e.g., \cite{Fernandez,Gonzalez,Hinkkanen,Jarvi,Lithner,Rocha,Osgood,Pommerenke,Pommerenke1984,Sugawa1998,Sugawa2003,Tukia}). Chen proved that for a planar domain $\Omega \subset \mathbb{C}$, its boundary $\partial \Omega$ is uniformly perfect if and only if the following two estimates hold simultaneously:
$$
K_{\Omega}(w) \asymp \delta_{\Omega}(w)^{-2}, b_{\Omega}(w)\asymp \delta_{\Omega}(w)^{-1}, w \to \partial \Omega.
$$
Based on this, Chen raised the following question:
\begin{question}
Is it true that $\partial \Omega$ is uniformly perfect equivalent to $K_{\Omega}(w) \asymp \delta_{\Omega}(w)^{-2}$? And is it true that $\partial \Omega$ is uniformly perfect equivalent to $b_{\Omega}(w) \asymp \delta_{\Omega}(w)^{-1}$?
\end{question}
In \cite{XiongZheng} and \cite{Zheng2025}, we discussed the boundary behavior of the Bergman kernel and provided an answer to the first question. In this paper, we wish to further investigate the boundary asymptotic behavior of $b_{\Omega}(z)$ and $d_{\Omega}(z_0,z)$, and to make progress toward answering the second question. To this end, we first need to examine the relation between the boundary behavior of $K_{\Omega}^{(1)}(z)$ and the boundary of the domain.

Our first main result is the following.
\begin{theorem}\label{K1equiv}
A domain $\Omega \subset \mathbb{C}$ has uniformly perfect boundary $\partial \Omega$ if and only if
$$
K^{(1)}_{\Omega}(w) \asymp \delta_{\Omega}(w)^{-4},\,\,\,w \to \partial \Omega.
$$
\end{theorem}

In \cite{Zheng2025}, we also extended the answer to the first question to a certain extent. We considered a generalization of uniform perfectness: a closed set $E \subset \mathbb{C}$ is said to be $h$-uniformly perfect (or simply said to satisfy condition $(U)_h$) if there exists $r_0>0$ such that for every $a \in E$ and $r \in (0,r_0)$,
$$
E \cap \{ z\in \mathbb{C}: h(r)\le |z-a| \le r  \} \ne \emptyset,
$$
where $h$ is a monotonically increasing function on $(0,r_0)$ with $h(r)<r$. In particular, when $E$ is bounded and $h(r)=Cr$ for some $0<C<1$, then $E$ is uniformly perfect in the classical sense. For $\alpha>1$ and $\beta>0$, we write
$$
h_{1,\alpha}(t)=t^{\alpha},\,\,\, h_{2,\beta}(t)=t\left( \log \frac{1}{t} \right)^{-\beta}.
$$
If there exists a constant $C>0$ such that $E$ is $h$-uniformly perfect with $h=Ch_{1,\alpha}$ or $h=Ch_{2,\beta}$, then we say respectively that $E$ satisfies condition $(U)_{1,\alpha}$ or condition $(U)_{2,\beta}$.

Generally speaking, as long as the set $E$ is compact and has no isolated points, one can always find a monotonically increasing continuous function $h$ such that $E$ satisfies condition $(U)_h$, and for any continuous function $f$ satisfying $f(r)>h(r)$ for all $r \in (0,r_0)$, the set $\partial \Omega$ must fail to satisfy condition $(U)_f$; that is, in some sense $\partial \Omega$ "just" satisfies condition $(U)_h$. Moreover, we may also require that $\frac{h(t)}{t}$ is monotonically increasing, and converges to 0 as $t \to 0$ if $E$ is not uniformly perfect. We shall explain the above statements in Section 2.

In \cite{Zheng2025}, we discussed the relationship between these two types of conditions and lower bound estimates for $K_{\Omega}(w)$. For $K^{(1)}_{\Omega}(w)$, correspondingly, we have the following result.

\begin{theorem} \label{lowerboundK1}
Let $\Omega \subset \mathbb{C}$ be a domain. 

$(1)$ If $\partial \Omega$ satisfies condition $(U)_{1,\alpha}$ for some $\alpha>1$, then as $w\to \partial \Omega$,
$$K^{(1)}_{\Omega}(w) \gtrsim \delta_{\Omega}(w)^{ -2-\frac{2}{\alpha}} \left( \log \frac{1}{\delta_{\Omega}(w)} \right)^{-1};$$

$(2)$ If $\partial \Omega$ satisfies condition $(U)_{2,\beta}$ for some $\beta>0$, then as $w\to \partial \Omega$,
$$K^{(1)}_{\Omega}(w) \gtrsim {\delta_{\Omega}(w)^{-4}\left(\log \frac{1}{\delta_{\Omega}(w)}\right)^{-2\beta}\left( \log \log \frac{1}{\delta_{\Omega}(w)}\right)^{-1}}.$$

The constants implicit in the above ``$\gtrsim$'' depend only on $\Omega$.
\end{theorem}

On the other hand, if $\partial \Omega$ does not satisfy condition $(U)_h$, then at certain points one obtains an upper estimate for $K^{(1)}_{\Omega}(w)$:

\begin{theorem}\label{upperboundK1}
Let $\Omega \subset \mathbb{C}$ be a domain, and suppose there exists a sequence of annuli 
$$A_n=\{z \in \mathbb{C}: h(r_n)<|z-a_n|<r_n\} \subset \Omega,$$ 
where $a_n \in \partial \Omega$, $r_n>0$, and $r_n \to 0$ as $n \to \infty$.

$(1)$ If $h(t)=Ch_{1,\alpha}(t)$ with $\alpha>1$, then there exist a numerical constant $C_0$ and a sequence of points $w_n \to \partial \Omega$ such that
$$
K^{(1)}_{\Omega}(w_n)\le C_0\cdot C^{\frac{2}{\alpha+1}} \cdot \delta_{\Omega}(w_n)^{-2-\frac{4}{\alpha+1}}.
$$

$(2)$ If $h(t)=Ch_{2,\beta}(t)$ with $\beta>0$, then there exist a numerical constant $C_0$ and a sequence of points $w_n \to \partial \Omega$ such that
$$
K^{(1)}_{\Omega}(w_n)\le C_0\cdot C \cdot{\delta_{\Omega}(w_n)^{-4}\left( \log \frac{1}{\delta_{\Omega}(w_n)}\right)^{-\beta}}.
$$
\end{theorem}

This shows that lower bounds on the boundary behavior of $K^{(1)}_{\Omega}(w)$ can conversely imply what kind of weak uniform perfectness the boundary satisfies:

\begin{corollary}\label{corollary}
$(1)$ If
$$
K^{(1)}_{\Omega}(w) \gtrsim \delta_{\Omega}(w)^{-2-\frac{4}{\alpha+1}},\,\,\,w \to \partial \Omega,
$$
then $\partial \Omega$ must satisfy condition $(U)_{1,\alpha}$.

$(2)$ If
$$
K^{(1)}_{\Omega}(w) \gtrsim {\delta_{\Omega}(w)^{-4}\left( \log \frac{1}{\delta_{\Omega}(w)}\right)^{-\beta}},\,\,\,w \to \partial \Omega,
$$
then $\partial \Omega$ must satisfy condition $(U)_{2,\beta}$.
\end{corollary}

If the boundary of the domain satisfies condition $(U)_{1,\alpha_1}$ but does not satisfy condition $(U)_{1,\alpha_2}$, where $1<\alpha_2 \le \alpha_1<2$, then by Theorem 1.4 of \cite{XiongZheng} we have
$$
K_{\Omega}(w) \gtrsim \delta_{\Omega}(w)^{-2}\left(  \log \frac{1}{\delta_{\Omega}(w)}  \right)^{-1}, 
$$
and combining this with Theorem \ref{upperboundK1} yields the existence of a sequence $w_n \to \partial \Omega$ such that
$$
b_{\Omega}(w_n)^2 \lesssim  \delta_{\Omega}(w_n)^{-2}\cdot \left( \delta_{\Omega}(w_n)^{\frac{2\alpha-2}{\alpha+1}}\log  \frac{1}{\delta_{\Omega}(w_n)}   \right)=o(\delta_{\Omega}(w_n)^{-2}).
$$
For the case where the boundary satisfies $(U)_{2,\beta_1}$ but not $(U)_{2,\beta_2}$, similarly there exists a sequence $w_n$ such that $b_{\Omega}(w_n)^2=o(\delta_{\Omega}(w_n)^{-2})$. This phenomenon seems to suggest that the answer to the second part of Chen's question is likely affirmative. Although we are not yet able to answer this question, we can obtain the following partial result:

\begin{theorem}\label{estimateb}
Let $\Omega$ be a bounded domain in $\mathbb{C}$ with weakly uniformly perfect boundary. We choose some functions $h_1(t), h_2(t)$, such that $\partial \Omega$ satisfies condition $(U)_{h_1}$ but does not satisfy condition $(U)_{h_2}$.  Furthermore, we require that $\frac{h_i(t)}{t} \ (i=1,2)$ be continuously increasing, satisfying
\begin{equation}\label{1.1}
\log \frac{h_1(t)}{t} \asymp \log  \frac{h_2(t)}{t} \to -\infty, \ \ \ t \to 0.
\end{equation}
If there exists a constant $0<C<2$ such that
\begin{equation}\label{1.2}
C\log  \frac{\tilde{h}_1(t)}{t} \le  \log  \frac{\tilde{h}_1\circ \tilde{h}_1(t)}{\tilde{h}_1(t)},
\end{equation}
where $\tilde{h}(t)=\frac{1}{5}h(t)$, then there must exist a sequence of points $w_n \to \partial \Omega$ in $\Omega$ such that
$$
b_{\Omega}(w_n) \ll \delta_{\Omega}(w_n)^{-1}.
$$
\end{theorem}

As mentioned earlier, we can always find a function $h$ such that $\partial \Omega$ "just" satisfies the condition $(U)_h$, so there exist appropriate $h_1, h_2$ such that (\ref{1.1}) holds. Therefore, the limiting condition here is (\ref{1.2}). Note that if
$$
h_1(t) \asymp t^{\alpha_1}, \ \ \  h_2(t)\asymp t^{\alpha_2},
$$ 
with $1<\alpha_2 \le \alpha_1<2$, then (\ref{1.2}) are satisfied; moreover, if
$$
h_1(t) \asymp t\left( \log^{\circ k} \frac{1}{t}   \right)^{-\beta_1}, \ \ \ h_2(t) \asymp t \left( \log^{\circ k} \frac{1}{t}   \right)^{-\beta_2},
$$ 
with $0< \beta_2 \le \beta_1$, then the theorem also applies, where the notation $f^{\circ k}$ denotes the $k$-fold composition of $f$ with itself, for $k \ge 1$; the theorem also applies to functions $h_i(t)$ of the form $Cte^{-\left( \log \frac{1}{t} \right)^{\gamma}}$ with $\gamma <1$. It can be seen that many domains with weakly uniformly perfect boundaries are applicable to this theorem. However, it is still difficult to answer Chen's question at present.

Next, we are concerned with lower bound estimates for the Bergman distance $d_{\Omega}(z_0,z)$, where $z_0$ is a fixed point in $\Omega$ and $z \to \partial \Omega$. It is known that when the boundary of the domain is uniformly perfect, we have
$$
d_{\Omega}(z_0, z) \gtrsim \log \frac{1}{\delta_{\Omega}(w)}.
$$
For the case of $h$-uniform perfectness, one approach is to translate condition $(U)_h$ into a hyperconvexity condition. From quantitative hyperconvexity conditions to lower bounds for the Bergman distance, there is already a standard procedure; see, e.g., \cite{Chen2017,Chen2023, ChenZheng}. When the boundary satisfies condition $(U)_{1,\alpha}$, it follows from \cite{Chen2023} that
$$
d_{\Omega}(z_0,z)\gtrsim \log \log \log \frac{1}{\delta_{\Omega}(w)}.
$$
If $\partial \Omega$ satisfies the condition $(U)_{2,\beta}$, then Lemma 3.1 in \cite{Chen2023} implies that there exists a continuous negative subharmonic function $\rho$ on $\Omega$ such that 
$$
-\rho(z) \le C_0\exp\left( C_1 \cdot \frac{\log \delta_{\Omega}(z)}{\log \log \frac{1}{\delta_{\Omega}(z)}} \right).
$$ 
Thus, using the same method, one can obtain
$$
d_{\Omega}(z_0,w)\gtrsim \frac{\log \frac{1}{\delta_{\Omega}(w)}}{\left( \log \log \frac{1}{\delta_{\Omega}(w)}\right)^2},
$$
the proof of which we leave to the reader.

However, the estimates obtained via this procedure seem to still have a certain gap from the optimal ones. For instance, on the two types of Zalcman-type domains constructed in \cite{XiongZheng}, the local behavior of the Bergman distance is found to be indeed superior to the above estimates. We wish to further investigate this phenomenon.

A natural approach is to estimate $b_{\Omega}(z)$ pointwise and then integrate to obtain its boundary behavior. Using Theorem 1.2, together with the trivial estimate $K_{\Omega}(w) \lesssim \delta_{\Omega}(w)^{-2}$, one can directly obtain a lower bound for $b_{\Omega}(w)$; unfortunately, substituting such an estimate into an integration along geodesics may not even yield Bergman completeness. This indicates that the estimate in Theorem \ref{lowerboundK1} only reflects the worst-case behavior of $b_{\Omega}(w)$, and Theorem \ref{upperboundK1} shows that such worst-case behavior indeed occurs. In other words, on domains with weakly uniformly perfect boundaries, the boundary behavior of $b_{\Omega}(z)$ fluctuates over a wide range, and only its better-behaved parts play a dominant role in the integration process for computing $d_{\Omega}(z_0,z)$. This also makes this approach applicable only to certain specific domains.

The proof of Theorem \ref{lowerboundK1} provides a method for obtaining lower bounds for $K^{(1)}_{\Omega}(w)$. However, in order to estimate the lower bound of $b_{\Omega}(z)$ as precisely as possible, one also needs to discuss upper bounds for the Bergman kernel; in this regard, see, e.g., \cite{BlockiZwonek2018}. For $w \in \overline{\Omega}$ and $R>0$, write $E(w,R)=\overline{D(w,R)}-\Omega$. We obtain the following result.

\begin{theorem}\label{upperboundK}
For any point $w$ in a domain $\Omega \subset \mathbb{C}$, if one chooses a real number $R>\delta_{\Omega}(w)$ such that
$$
\mathrm{Cap}(E(w,R))<2\delta_{\Omega}(w),
$$
then one necessarily has
$$
K_{\Omega}(w)\le \frac{18}{\pi}\cdot \frac{1}{\delta_{\Omega}(w)^2}\cdot \left( \frac{1}{\log \frac{2\delta_{\Omega}(w)}{\mathrm{Cap}(E(w,R))}}+\frac{1}{\log \frac{2R}{\delta_{\Omega}(w)}}   \right).
$$
\end{theorem}

For each point $w$, by choosing an appropriate $R=R(w)$, one can give an asymptotic upper estimate for $K_{\Omega}(w)$. In practice, one should choose a suitable $R$ to balance the two terms inside the parentheses on the right-hand side.

We can make a simple comparison between this result and Theorem 1.4 of \cite{XiongZheng}. For $w\in \Omega$, choose $w' \in \partial \Omega$ such that $|w-w'|=\delta_{\Omega}(w)$. 
\begin{itemize}
\item 
Assume that for every $w \in \overline{\Omega}$  and every $r \in (0, \mathrm{diam}(\Omega)]$, one always has
$$
\mathrm{Cap}(E(w,r)) \lesssim r^{\alpha}, \ \ \  \alpha>0.
$$
In this case, we may take $R= \delta_{\Omega}(w)^{\alpha'}$, where $\alpha'>0$ satisfies $ \frac{1}{\alpha}< \alpha'<1$. Since
$$
\mathrm{Cap}(E(w,R)) \lesssim R^{\alpha}\asymp \delta_{\Omega}(w)^{\alpha \alpha'}<2\delta_{\Omega}(w),
$$
Theorem \ref{upperboundK} then yields
$$
K_{\Omega}(w) \lesssim \frac{1}{\delta_{\Omega}(w)^2}\left( \frac{1}{\log \frac{2\delta_{\Omega}(w)}{\delta_{\Omega}(w)^{\frac{1}{2}(\alpha+1)}}} +\frac{1}{\log \frac{ \delta_{\Omega}(w)^{\frac{1}{2}(1+\frac{1}{\alpha})} }{\delta_{\Omega}(w)}}    \right) \lesssim \frac{1}{\delta_{\Omega}(w)^2 \log \frac{1}{\delta_{\Omega}(w)}}.
$$

\item
Assume that for every $w \in \overline{\Omega}$ and every $r \in (0, \mathrm{diam}(\Omega)]$, one always has
$$
\mathrm{Cap}(E(w,r)) \lesssim r \left(  \log \frac{1}{r} \right)^{-\beta},\ \ \  \beta>0.
$$
In this case, we may take
$$R = \delta_{\Omega}(w)\left( \log \frac{1}{\delta_{\Omega}(w)}  \right)^{\beta'},$$ 
where $0<\beta'<\beta$. Then we always have
\begin{eqnarray*}
\mathrm{Cap}(E(w,R)) &\lesssim&  R \left( \log \frac{1}{R}  \right)^{-\beta}\\
&\lesssim&  \delta_{\Omega}(w) \left( \log \frac{1}{\delta_{\Omega}(w)} \right)^{\beta'} \left( \log \frac{1}{\delta_{\Omega}(w)} -\frac{\beta}{2} \log \log \frac{1}{\delta_{\Omega}(w)}  \right)^{-\beta}\\
&<&2\delta_{\Omega}(w).
\end{eqnarray*}
Therefore we obtain
$$
K_{\Omega}(w) \lesssim \frac{1}{\delta_{\Omega}(w)^2 \log \log \frac{1}{\delta_{\Omega}(w)}}.
$$ 
\end{itemize}
Thus the estimate in Theorem \ref{upperboundK} is already quite precise.

As an application of Theorems \ref{lowerboundK1} and \ref{upperboundK}, on some specific domains we can obtain the following result.

\begin{theorem}\label{lowerboundd}
Let $\Omega \subset \mathbb{C}$ be a domain with $0 \in \partial \Omega$, and suppose $\partial \Omega$ satisfies condition $(C)_h$. Assume that there exist a constant $c \in (0,1)$ and a sequence of annuli
$$
A_k:= \{ z \in \mathbb{C}; h(r_k) <|z|<r_k  \} \subset \Omega, \,\,\,k=1,2,\cdots, 
$$
such that $ch(r_k) \le r_{k+1} \le h(r_k)$ for each $k \in \mathbb{N}^{+}$. 

$(1)$ When $h(t)=Ct^{\alpha}$, suppose that there exists $\alpha' \in (1, \frac{1}{2-\alpha})$ such that for each $k$,
$$
\mathrm{Cap}(\overline{D(0,h(r_k))} \setminus \Omega) \lesssim r_k^{ \alpha'}.
$$
Then the following estimate holds:
$$
d_{\Omega}(z, z_0) \gtrsim \log \log \frac{1}{|z|},\,\,\,z \to 0.
$$

$(2)$ When $h(t)=Ct\left( \log \frac{1}{t}  \right)^{-\beta}$, suppose that there exists $\beta' \in (0, 2\beta)$ such that for each $k$,
$$
\mathrm{Cap}(\overline{D(0,h(r_k))} \setminus \Omega) \lesssim r_k \left( -\log \frac{1}{r_k}  \right)^{-\beta'},
$$
then the following estimate holds:
$$
d_{\Omega}(z, z_0) \gtrsim  \frac{\log \frac{1}{|z|}}{\log \log \frac{1}{|z|}},\,\,\,z \to 0.
$$
\end{theorem}

It is easy to verify that the Zalcman-type domains $\Omega_{1,\alpha}$ and $\Omega_{2,\beta}$ constructed in \cite{XiongZheng} satisfy the hypotheses of Theorem \ref{lowerboundd}.

\section{Preliminaries}

\subsection{Weak uniform perfectness}

For a closed set $E \subset \mathbb{C}$ and $a \in E$, we define the following function:
$$
h_a(r)=\inf \{ r': A(a; r', r) \subset   E^c   \},  \ \ \ \forall r \in (0, \mathrm{diam}(\Omega)), 
$$
where
$$
A(a; r',r):=\{ z \in \mathbb{C}: r'<|z-a|<r  \},
$$
and if $r'=r$ we regard $A(a; r',r)=\emptyset$, so that the definition of $h_a(r)$ is always meaningful and we always have $h_a(r) \le r$. It is also easy to see that $h_a(r)$ is monotonically increasing in $r$.

Now set
$$
h_E(r)=\inf_{a \in E}  {h_a(r)}, \ \ \ \forall r \in (0, \mathrm{diam}(\Omega)], 
$$
Clearly $0 \le h_E(r) \le r$, and since each $h_a(r)$ is monotonically increasing, it follows that $h_E(r)$ is also monotonically increasing.

\begin{proposition}\label{prop2.1}
Let $E \subset \mathbb{C}$ be a closed set.

$(1)$ $E$ satisfies condition $(U)_{h_E}$;

$(2)$ If for some $r \in (0, \mathrm{diam}(E)]$ we have $h_E(r)<r'$, then there exists $a \in E$ such that
$$\{ z: r'<|z-a|<r  \} \subset E^c.$$
\end{proposition}
\begin{proof}
$(1)$ Indeed, for any $a \in E$ and $r\in (0, \mathrm{diam}(E)]$, we have
$$
\{  z: h_E(r) \le |z-a| \le r \} \cap E \ne \emptyset.
$$
Otherwise, there would exist $a\in E$ and $r\in (0, \mathrm{diam}(E)]$ such that
$$
\{ z: h_E(r) \le |z-a| \le r  \} \subset E^c.
$$ 
Since $\{ z: h_E(r) \le |z-a| \le r  \}$ is compact, there must exist $\varepsilon>0$ such that
$$
\{ z: h_E(r)-\varepsilon \le |z-a| \le r  \}\subset E^c,
$$
hence
$$
h_E(r)\le h_a(r) \le h_E(r)-\varepsilon,
$$
which is a contradiction.

$(2)$ If $r' \ge r$, then trivially
$$
\{ z: r'<|z-a|<r \}=\emptyset  \subset E^c.
$$
If $h_E(r)<r'<r$, then there must exist $a \in E$ such that
$h_a(r)<r'<r,$
and by the definition of $h_a$ it immediately follows that
$$
\{ z: r'<|z-a|<r \} \subset  E^c.
$$
\end{proof}

However, this $h_E$ may fail to be continuous, and it may happen that $h_E(r) \ll r$ for many $r$ while at the same time $h_E(r)=r$ at many other points, and $\frac{h_E(r)}{r}$ need not be monotonic either. Therefore it is still not suitable for characterizing condition $(U)_h$.

Set
$$
u(r)=\frac{h_E(r)}{r}, \ \ \ r \in (0, \mathrm{diam}(E)],
$$
and define
$$
\tilde{u}(r)=\inf \{u(t):  {t \in [r,\mathrm{diam}(E)]}  \}.
$$
Then $0 \le \tilde{u}(r) \le u(r)$, and $\tilde{u}(r)$ is monotonically increasing. Moreover, for weakly uniformly perfect sets one necessarily has
$$
\tilde{u}(r) \to 0, \ \ \ r \to 0,
$$
since otherwise we would have $h(t) \gtrsim t$, contradicting the assumption that $E$ is not uniformly perfect.

\begin{lemma}\label{lemma2.2}
Let $E\subset \mathbb{C}$ be compact. Then the following statements are equivalent:

$(1)$ $\tilde{u}(r)=0$ for some $r \in (0, \mathrm{diam}(E)]$;

$(2)$ $h_E(r)=0$ for some $r \in (0, \mathrm{diam}(E)]$;

$(3)$ $E$ has an isolated point.
\end{lemma}
\begin{proof}
First note that $(3) \Rightarrow (2) \Rightarrow (1)$ are obvious, so it remains to prove $(1) \Rightarrow (3)$.
From
$$
\tilde{u}(r)=\inf \{u(t) : {t \in [r, \mathrm{diam}(E)]} \} =0,
$$
it follows that for each $n \in \mathbb{N}$, there exists $x_n \in [r, \mathrm{diam}(\Omega)]$ such that $u(x_n)<\frac{1}{n}$, i.e., $h_E(x_n)<\frac{x_n}{n}$. Hence there exists $a_n \in E$ such that
$$
h_{a_n}(x_n)<\frac{x_n}{n},
$$
that is,
$$
A\left(a_n; \frac{x_n}{n}, x_n\right):=\left\{ z: \frac{x_n}{n}<|z-a_n|<x_n   \right\} \subset E^c.
$$
Since $E$ is compact, the sequence $\{ a_n\}$ has an accumulation point $a \in E$; we may assume without loss of generality that $a_n \to a$. For any $w$ sufficiently close to $a$, there exists $n$ such that
$$
|a_n-a|+\frac{x_n}{n} \le |a_n-a|+\frac{\mathrm{diam}(E)}{n}<|w-a|,
$$
hence
$$
|w-a_n|>\frac{x_n}{n},
$$
so
$$
w \in A\left(a_n; \frac{x_n}{n}, x_n\right) \subset E^c.
$$
By the arbitrariness of $w$, it follows that $a$ is an isolated point of $E$, i.e., $(1) \Rightarrow (3)$. This completes the proof.
\end{proof}

In order to find a suitable function for characterizing the weak uniform perfectness of $E$, we also need the following lemma.

\begin{lemma}\label{lemma2.3}
Let $g(t)$ be a monotonically increasing function on the interval $[a,b]$ with $0 \le g(a) \le g(b)<+\infty$. Then there exists a continuous increasing function $\tilde{g}(t)$ on $[a,b]$ such that $\tilde{g}(t) \le g(t)$, and $\tilde{g}(a)=g(a)$, $\tilde{g}(b)=g(b-)$.
\end{lemma}
\begin{proof}
Since $g$ has at most countably many discontinuity points, let us denote by $\{ a_i \}$ its discontinuity points other than possibly $a$ and $b$. On $[a,b]$, we represent $g$ as
$$
g(t)=g_c(t)+\sum_{i=1}^{\infty} \Delta_i(t),
$$
where $g_c$ is a continuous function and $\Delta_i(t)$ is the jump function at $a_i$ (cf. \cite{Stein}, Lemma 3.13, e.g.). For each $\Delta_i(t)$, provided $a_i \ne b$, it is easy to find a continuous increasing function $\phi_i(t) \le \Delta_i(t)$ such that
$$
\phi_i(a)= \Delta_i(a)=0, \ \ \   \phi_i(b)= \Delta_i(b)=g(a_i+)-g(a_i-).
$$
If $a_i=b$, we simply set $\phi_i \equiv 0$.
Now take
$$\tilde{g}(t)=g_c(t)+\sum_i \phi_i(t).$$
Clearly $\sum_i \phi_i(t)$ converges uniformly, so it is easy to verify that $\tilde{g}$ is continuous and increasing on $[a,b]$, and satisfies
$$
\tilde{g}(a)=g(a), \ \ \ \tilde{g}(b)=g(b-).
$$
\end{proof}

\begin{proposition}\label{prop2.4}
Let $E \subset \mathbb{C}$ be a compact set which is not uniformly perfect and has no isolated points. Then there exists a continuous increasing function $h(r)$ such that:

$(1)$ $h_E(r) \ge h(r)$;

$(2)$ the function $s(r):=\frac{h(r)}{r}$ is continuous and increasing;

$(3)$ for any continuous increasing function $f(r)$ satisfying $f(r)>h(r)$ for all $r \in (0, r_0)$ for some $r_0$, $E$ does not satisfy condition $(U)_f$.
\end{proposition}
\begin{proof}

Let $u$ and $\tilde{u}$ be as defined above. The function $\tilde{u}$ need not be continuous, but by monotonicity it has at most countably many discontinuities. Hence we can choose a sequence $r_n$ of continuity points of $\tilde{u}$ decreasing monotonically to zero. On $[r_n, \mathrm{diam}(\Omega)]$, there must exist a sequence $\{t_{n,m}\}_{m \in \mathbb{N}}$ such that
$$
u(t_{n,m}) \to \tilde{u}(r_n), \ \ \ m \to +\infty.
$$
Take an accumulation point of $\{ t_{n,m} \}_{m \in \mathbb{N}}$ and denote it by $t_n$; without loss of generality, assume $t_{n,m} \to t_n$ as $m \to +\infty$.

Claim: $t_n$ must converge to $0$.
Indeed, otherwise there would be a subsequence 
$$t_{n_k} \to \tilde{t} \in (0, \mathrm{diam}(\Omega)].$$ 
Then for sufficiently small $\varepsilon$ and sufficiently large $k$ such that $r_{n_k}<\tilde{t}-\varepsilon$, we have $t_{n_k} \in (\tilde{t}-\varepsilon, \tilde{t}+\varepsilon)$; consequently, for $m$ sufficiently large, $t_{n_k, m} \in (\tilde{t}-\varepsilon, \tilde{t}+\varepsilon)$. Then we have
\begin{equation}
0<\tilde{u}(r_{n_k}) \le \inf_{t \in (\tilde{t}-\varepsilon, \tilde{t}+\varepsilon)}u(t) \le u(t_{n_k,m}) \to \tilde{u}(r_{n_k}), \ \ \ m \to +\infty.
\end{equation}
Here $\tilde{t}$ is fixed, meaning that $\tilde{u}(r_{n_k})$ is a positive constant independent of $k$, contradicting the fact that $\tilde{u}(t)=o(t)$! Thus the Claim holds.

Now assume without loss of generality that $t_n$ is monotonically decreasing to $0$.
On $[t_{n+1}, t_n]$, by the monotonicity of $\tilde{u}$, we clearly have
\begin{equation}
\tilde{u}(r_{n+1}) \le \tilde{u}(t_{n+1}).
\end{equation}
Moreover, at $t_n$, we can show that
\begin{equation}
\tilde{u}(r_n) =\tilde{u}(t_n-).
\end{equation}
In fact, if $t_n=r_n$, then since $r_n$ is a continuity point of $\tilde{u}$, (2.3) holds trivially; if $r_n<t_n$, then this means that the infimum of $u$ on $[r_n, \mathrm{diam}(\Omega)]$ is attained in arbitrarily small neighborhoods of $t_n$, that is, for every $\varepsilon>0$,
$$
\tilde{u}(r_n)=\tilde{u}(t_n-\varepsilon),
$$
so (2.3) still holds.

Next, consider the sequence of points $\{ (t_n, \tilde{u}(r_n)) \}$. We wish to construct a continuous increasing function that "connects" these points. More precisely, we construct a continuous increasing function $s(r)$ on $[t_{n+1}, t_n]$ such that
$$
s(r) \le \tilde{u}(r), \ \ \ s(r_n)=\tilde{u}(r_n), \ \ \ s(r_{n+1})=\tilde{u}(r_{n+1}).
$$
By (2.2) and (2.3), the lemma shows that this is possible. Patching together the definitions of $s$ on each interval, we obtain a continuous increasing function on $(0, t_0)$.

Finally, set
$$
h(r)=r s(r).
$$
Clearly $h(r) \le r \tilde{u}(r) \le h_E(r)$, so $E$ satisfies condition $(U)_h$. Moreover, $s(r)$ is continuous and increasing. Now suppose there is another continuous function $f(r)$ such that $f(r)>h(r)$ for every $r \in (0,t_1)$. Then in particular,
$$
f(t_n)>t_n \tilde{u}(r_n)=t_n\tilde{u}(t_n-).
$$
By continuity of $\frac{f(t)}{t}$, in a sufficiently small left-neighborhood $(t_n-\varepsilon_n, t_n]$ of each $t_n$, one can find $t_n'$ such that
$$f(t_n')>t_n' \tilde{u}(t_n').$$
By the choice of $t_n$, we have
$$
\tilde{u}(t_n')=\inf \{ u(t): t_n' \le t < \mathrm{diam}(\Omega)  \} =\inf \{ u(t): t_n' \le t <t_n+\varepsilon_n \},
$$
so there exists some $t_n'' \in (t_n-\varepsilon_n, t_n+\varepsilon_n)$ such that
$$
f(t_n'')>h_E(t_n'').
$$
Combining this with Proposition 2.1 (2), we conclude that $E$ cannot satisfy condition $(U)_f$.
\end{proof}

\subsection{Weak uniform perfectness and logarithmic capacity}

An important property of uniform perfectness is that it can be equivalently characterized by logarithmic capacity (\cite{Pommerenke}, Theorem 1). There is an analogous characterization for $h$-uniform perfectness. We say that a closed set $E \subset \mathbb{C}$ satisfies condition $(C)_h$ if there exists a constant $r_0>0$ such that for every $a \in E$ and $r \in (0,r_0)$,
$$
\mathrm{Cap}(E(a,r))\ge h(r),
$$
where $E(a,r):=\overline{D(a,r)} \setminus \Omega$, and $h$ is a monotonically increasing function on $(0,r_0)$ satisfying $h(r)<r$. In particular, if there exists a constant $C>0$ such that $h(t)$ is of the form $Ct^{\alpha}$ or $Ct\left(\log \frac{1}{t}\right)^{-\beta}$, we say respectively that $E$ satisfies condition $(C)_{1,\alpha}$ or $(C)_{2,\beta}$, where $\alpha>1$ and $\beta>0$. We also define
$$
g_{E}(r):= \inf_{a \in E} \mathrm{Cap}(E \cap \overline{D(a,r)}),
$$
then clearly $E$ always satisfies condition $(C)_{g_E}$. The following property is rather obvious.

\begin{proposition}\label{prop2.5}
Let $E\subset \mathbb{C}$ be compact. If there exists $t_0 \in (0, \mathrm{diam}(E)]$ such that $g_E(t_0) = 0$, then $E$ must be locally polar in the sense that there exist $a \in E$ and $r >0$ such that 
$$
\mathrm{Cap}(E \cap \overline{D(a,r)})=0.
$$
\end{proposition}
\begin{proof}
Since $g_E \ge 0$ is monotonically increasing, the above assumption implies that
$$
g_E(t) \equiv 0, \quad t \in (0,t_0].
$$
Take $r=\frac{t_0}{2}$; then $g_E(2r)=g_E(t_0)=0$. Hence, for every $n \in \mathbb{N}$, there exists $a_n \in E$ such that
$$
0\le \mathrm{Cap}(E \cap \overline{D(a_n,2r)})<\frac{1}{n}.
$$
Since $E$ is compact, the sequence $\{ a_n \}$ has an accumulation point; assume without loss of generality that $a_n \to a \in E$. Then for $n$ sufficiently large,
$$
E \cap \overline{D(a, r)} \subset E \cap \overline{D(a_n, 2r)},
$$
so
$$
0 \le \mathrm{Cap}\left(E \cap \overline{D(a, r)} \right) \le \mathrm{Cap}\left( E \cap \overline{D(a_n, 2r)} \right) <\frac{1}{n}.
$$
Letting $n \to \infty$, we obtain
$$
\mathrm{Cap}\left(E \cap \overline{D(a, r)} \right)=0.
$$
This completes the proof.
\end{proof}

In this paper we are mainly concerned with the case where $E$ is the boundary of some domain $\Omega$. If some local part of $E$ is polar, then $K_{\Omega}(w), K^{(1)}_{\Omega}(w)$ and $b_{\Omega}(w)$ are at least bounded in that local part, which is a trivial case. Hence in what follows we shall always assume that $g_E(t)>0$.

Concerning the relationship between conditions $(C)_h$ and $(U)_h$, we made a mistake in the proof of Theorem 1.3 in \cite{XiongZheng}: the definition of the mapping $\omega$ was problematic, and this error led to an incorrect estimate of the diameter of $E_k$. More precisely, we neglected the definition of $\phi_0(a)$. In fact, the sequence $\{s_k\}$ should start from $s_0$, with
$$
\phi_0(a) \in \partial \Omega \cap \{  z:  5 s_1 \le  |z-a| \le s_0  \}.
$$
Then one should have $E_k \subset \overline{D(a,2s_0)} \setminus \Omega$, so that (4.6) in \cite{XiongZheng} should be corrected to
$$
\log \mathrm{Cap} \left(  \overline{D(a,2s_0)} \setminus \Omega   \right) \ge \sum_{l=0}^{\infty} \frac{ \log s_{l+1}}{2^{l+1}}.
$$
In particular, after correcting the above error, one obtains $(U)_{1,\alpha} \Rightarrow (C)_{1, \frac{\alpha}{2-\alpha}}$ and $(U)_{2,\beta} \Rightarrow (C)_{2,2\beta}$, where $\alpha \in (1,2)$ and $\beta>0$. For convenience in what follows, we restate the theorem as follows:

\begin{theorem}[\cite{XiongZheng}] \label{thm2.6}
Let $E \subset \mathbb{C}$ be a closed set, and let $h$ be a monotonically increasing function on $(0,1)$ satisfying $0 \le h(t)< t$.

$(1)$ If $E$ satisfies condition $(C)_{h}$, then it necessarily satisfies condition $(U)_{h}$.

$(2)$ If $E$ satisfies condition $(U)_{h}$, then it necessarily satisfies condition $(C)_{g}$, where
$$g(r)=\exp \left( \sum_{k=1}^{\infty}\frac{\log \tilde{h}^{\circ k}\left(  \frac{r}{2}   \right)}{2^k}   \right),
$$
and $\tilde{h}(t)=\frac{1}{5}h(t)$. In particular,  $(U)_{1,\alpha} \Rightarrow (C)_{1, \frac{\alpha}{2-\alpha}}$ and $(U)_{2,\beta} \Rightarrow (C)_{2,2\beta}$, where $\alpha \in (1,2)$ and $\beta>0$. 
\end{theorem}

Here (2) follows by correcting the proof in \cite{XiongZheng} as described above. We briefly sketch the proof of (1).

\begin{proof}[Proof of Theorem \ref{thm2.6} $(1)$]
Suppose on the contrary that condition $(U)_{h}$ fails. Then, for any $r_0>0$, there exist $a \in E$ and $r \in (0,r_0)$ such that
\[
 \{ z\in \mathbb{C}: {h}(r)  \leq |z-a| \leq r \} \subset E^c.
\]
Since $\{ z\in \mathbb{C}: {h}(r)  \leq |z-a| \leq r \}$ is compact, there exists some $\varepsilon>0$ such that
\[
\{   z\in \mathbb{C}: {h}(r)-\varepsilon  \leq |z-a| \leq r+\varepsilon  \} \subset E^c.
\]
Thus, $E \cap \overline{D(a,r)} \subset \overline{D(a, h(r)-\varepsilon)}$, which implies
\[
h(r) \le \mathrm{Cap}(E \cap \overline{D(a,r)})<\mathrm{Cap}(\overline{D(a, h(r)-\varepsilon)})=h(r)-\varepsilon.
\]
But this is a contradiction, since $E$ satisfies condition $(C)_h$.
\end{proof}
Note that when $t\ll 1$, the series
$$\sum_{k=1}^{\infty} \frac{ \log \tilde{h}^{\circ k}\left( t \right)}{2^{k}}$$
is a series of negative terms and converges in the generalized sense, so the function $g$ is always well-defined. It is easy to see from the computation that if $h(t) \lesssim t^2$, then the corresponding $g(t)$ vanishes identically. Moreover, the functions $g$ and $\tilde{h}$ satisfy the following quantitative relation:
\begin{eqnarray*}
g \left( 2\tilde{h}\left( t \right)\right)&=& \exp \left( \sum_{k=1}^{\infty} \frac{ \log \tilde{h}^{\circ (k+1)}\left( t \right)}{2^{k}}  \right) \\
&=& \exp \left( 2\sum_{k=1}^{\infty} \frac{ \log \tilde{h}^{\circ k}\left( t  \right)}{2^{k}}-\log \tilde{h}\left( t \right)  \right) \\
&=& \frac{g(2t)^2}{\tilde{h}(t)}.
\end{eqnarray*}
That is,
$$
\tilde{h}\left( t \right)\cdot g(2\tilde{h}\left( t \right))=g(2t)^2.
$$

\section{Computation of the Bergman metric on an annulus}

In this section, we estimate the value of the Bergman metric at $z=1$ on the annulus $A_R:=\{ \zeta \in \mathbb{C}; \frac{1}{R} <|\zeta|< R   \}$, where $R$ is sufficiently large. First, it is easy to verify that $A^2(\Omega)$ has an orthonormal basis $\{  \phi_n(z)  \}_{n=-\infty}^{\infty}$, where
$$
\phi_{-1}(z)=\frac{1}{\sqrt{4\pi \log R}}\cdot z^{-1}, \,\,\,\phi_{n}(z)=\sqrt{\frac{n+1}{\pi (R^{2n+2}-R^{-2n-2})}}\cdot z^{n}\,\,\, (n \ne -1).
$$
Thus it is readily seen that the Bergman kernel on $A_R$ can be expressed as
$$
K_{A_R}(z)=\frac{1}{4\pi |z|^2 \log R}+\frac{1}{\pi |z|^2}\sum_{n=1}^{\infty}\frac{n(|z|^{2n}+|z|^{-2n})}{R^{2n}-R^{-2n}}.
$$
It depends only on $|z|$, so we may write $K(|z|):=K_{A_R}(z)$. Also set
$$
S(|z|)=\sum_{n=1}^{\infty}\frac{n(|z|^{2n}+|z|^{-2n})}{R^{2n}-R^{-2n}}.
$$
A direct computation gives
\begin{eqnarray}\label{chapter3, 3.1}
K(1)&=&\frac{1}{4\pi \log R}+\frac{1}{\pi}\sum_{n=1}^{\infty} \frac{2n}{R^{2n}-R^{-2n}}=\frac{1}{4\pi \log R}+O(R^{-2}),\\
\notag K(r)&=&\frac{1}{\pi} \left( \frac{1}{4r^2\log R}+\frac{1}{r^2}S(r)    \right),\\
\notag K'(r)&=&\frac{1}{\pi} \left( -\frac{1}{2r^3\log R}-\frac{2}{r^3}S(r)+\frac{1}{r^2}S'(r)        \right),\\
\notag K''(r)&=&\frac{1}{\pi} \left(  \frac{3}{2r^4 \log R}+\frac{6}{r^4}S(r)-\frac{4}{r^3}S'(r)+\frac{1}{r^2}S''(r)   \right),
\end{eqnarray}
where
\begin{eqnarray*}
S'(r)&=&2\sum_{n=1}^{\infty} \frac{n^2(r^{2n-1}-r^{-2n-1})}{R^{2n}-R^{-2n}}, \\
S''(r)&=&2\sum_{n=1}^{\infty} \frac{ n^2\left( (2n-1)r^{2n-2}+(2n+1)r^{-2n-2}     \right) }{R^{2n}-R^{-2n}}.
\end{eqnarray*}
In particular,
\begin{eqnarray*}
S(1)&=&\sum_{n=1}^{\infty}\frac{2n}{R^{2n}-R^{-2n}}=\frac{2}{R^2}+O(R^{-4}),\\
S'(1)&=&0,\\
S''(1)&=&8\sum_{n=1}^{\infty}\frac{n^3}{R^{2n}-R^{-2n}}=\frac{8}{R^2}+O(R^{-4}).
\end{eqnarray*}
Consequently,
\begin{equation}\label{chapter3, 3.2}
K'(1)=-\frac{1}{2\pi \log R}-\frac{4}{\pi R^2}+O(R^{-4}),
\end{equation}
\begin{equation}\label{chapter3, 3.3}
K''(1)=\frac{3}{2\pi \log R}+\frac{20}{\pi R^2}+O(R^{-4}).
\end{equation}

Next, we wish to compute the Bergman metric $b_{A_R}(z)^2 \mathrm{d}z \otimes \mathrm{d}\overline{z}$. Note that the map $T: z \to \frac{1}{z}$ is a holomorphic automorphism of $A_R$ onto itself. Therefore, by the transformation property of the Bergman kernel,
$$
K_{A_R}\left(\frac{1}{z}\right)\cdot \left|\frac{-1}{z^2}\right|^2=K_{A_R}(z),
$$
that is,
\begin{equation}\label{chapter3, 3.4}
K_{A_R}\left(\frac{1}{z}\right)=|z|^4K_{A_R}(z).
\end{equation}
Write
$$
f(r)=\log K(r), \quad r=|z|.
$$
Then by the definition of the Bergman metric, using the polar coordinate expression of the Laplacian and noting that $f$ depends only on the radial variable, we have
$$
b_{A_R}(z)^2=\frac{\partial^2}{\partial z \partial \overline{z}} \log K_{A_R}(z) =\frac{1}{4}\Delta f(r)=\frac{1}{4}\left( f''(r)+\frac{1}{r}f'(r) \right).
$$
First, taking logarithms on both sides of (\ref{chapter3, 3.4}), we obtain
$$
f\left(\frac{1}{r}\right)=4\log r +f(r).
$$
Differentiating both sides gives
$$
f'\left( \frac{1}{r} \right)\cdot \left( -\frac{1}{r^2} \right)=\frac{4}{r}+f'(r).
$$
Setting $r=1$, we get $f'(1)=-2$, hence
\begin{equation}\label{chapter3, 3.5}
b_{A_R}(1)^2=\frac{1}{4} \left( f''(1)-2  \right).
\end{equation}
Now we compute $f''(1)$. Since $f(r)= \log K(r)$, we have
$$
f''(r)=\left( \frac{K'(r)}{K(r)}  \right)'=\frac{K''(r)K(r)-K'(r)^2}{K(r)^2}.
$$
Substituting (\ref{chapter3, 3.1}), (\ref{chapter3, 3.2}), and (\ref{chapter3, 3.3}) into the above, a straightforward calculation yields
\begin{equation*}
f''(1)= 2+O\left(\frac{\log R}{R^2}\right).
\end{equation*}
Plugging this into (\ref{chapter3, 3.5}), there exists a numerical constant $C_0$ such that
\begin{equation}\label{chapter3, 3.6}
b_{A_R}(1)^2=O\left( \frac{\log R}{R^2}    \right),\,\,\,K^{(1)}_{A_{R}}(1)=b_{A_R}(1)^2\cdot K_{A_R}(1)=O \left( \frac{1}{R^2} \right).
\end{equation}
The constants implicit in the above $O$-notation are absolute numerical constants.

\section{Proofs of Theorem \ref{K1equiv} and Theorem \ref{upperboundK1}}

\begin{proof}[Proof of Theorem \ref{K1equiv}]
If $\partial \Omega$ is uniformly perfect, then by \cite{Chen2013} we have
$$
K_{\Omega}(w) \asymp \delta_{\Omega}(w)^{-2},\,\,\,b_{\Omega}(w) \asymp \delta_{\Omega}(w)^{-1},
$$
and hence obviously
$$
K^{(1)}_{\Omega}(w)=b_{\Omega}(w)^2\cdot K_{\Omega}(w) \asymp \delta_{\Omega}(w)^{-4}.
$$

For the converse direction, we argue by contradiction. Suppose that the boundary of the domain $\Omega \subset \mathbb{C}$ is not uniformly perfect. Then for any $C\in (0,1)$, there exist a point $a \in \partial \Omega$ and $r \in (0,\mathrm{diam}(\Omega))$ such that
$$
\partial \Omega \cap \tilde{A}=\emptyset,
$$
where
$$
\tilde{A}=\{ z\in \mathbb{C}: Cr < |z-a| < r  \}.
$$
It is clear that as long as $C<\frac{1}{2}$, we must have $\tilde{A}\subset \Omega$. In particular, take a sequence $C_n \to 0$; then there corresponds a sequence of annuli
$$
\tilde{A}_n=\{ z\in \mathbb{C}: C_nr_n <|z-a_n|<r_n  \} \subset \Omega.
$$
Take $z_n \in \tilde{A_n}$ such that $|z_n-a_n|=\sqrt{C_n}r_n$. Then it is geometrically evident that
\begin{equation} \label{chapter4, 4.1}
\frac{\sqrt{C_n}r_n}{2} \le \left(\sqrt{C_n}-C_n \right)r_n= \delta_{\tilde{A}_n}(z_n)\le \delta_{\Omega}(z_n) \le \sqrt{C_n}r_n.
\end{equation}
Consider the conformal mapping
$$
T: \tilde{A}_n \to A_{R_n}, \quad z \mapsto \frac{z-a_n}{\sqrt{C_n}r_n},
$$
where $R_n=\frac{1}{\sqrt{C_n}}$ and $A_{R_n}=\{ z\in \mathbb{C}: \frac{1}{R_n}<|z|<R_n  \}$. By the conformal invariance of the Bergman metric, together with (\ref{chapter3, 3.6}) and (\ref{chapter4, 4.1}), we obtain
\begin{eqnarray*}
K_{\Omega}^{(1)}(z_n)&\le& K^{(1)}_{\tilde{A}_n}(z_n)\\
&=&b_{\tilde{A}_n}(z_n)^2 \cdot K_{\tilde{A}_n}(z_n)\\
&=&K^{(1)}_{A_{R_n}}(T(z_n))|T'(z_n)|^4\\
&=& \frac{1}{C_n^2r_n^4}\cdot O\left(\frac{1}{R_n^2}\right)\\
&\le & O(C_n)\cdot \delta_{\Omega}(z_n)^{-4}
\end{eqnarray*}
Since $O(C_n) \to 0$ as $n \to \infty$, this contradicts $K_{\Omega}^{(1)}(w) \asymp \delta_{\Omega}(w)^{-4}$.
\end{proof}

\begin{proof}[Proof of Theorem \ref{upperboundK1}]
Suppose that for the domain $\Omega$, for every $\tilde{r}>0$, there exist $a \in \partial \Omega$ and $r \in (0,\tilde{r})$ such that
$$
\partial \Omega \cap \tilde{A}=\emptyset,
$$
where
$$
\tilde{A}=\{z \in \mathbb{C}: h(r)<|z-a|<r \}.
$$
Since $\tilde{r}$ can be chosen arbitrarily small, we always have $\tilde{A} \subset \Omega$ in that case. In particular, taking a sequence $\tilde{r}_n \to 0$, there exists a sequence of annuli
$$
\tilde{A}_n=\{z \in \mathbb{C}: h(r_n)<|z-a_n|<r_n \} \subset \Omega,
$$
with $0<r_n<\tilde{r}_n$.

Let $z_n \in \tilde{A_n}$ be such that $|z_n-a_n|=\sqrt{r_nh(r_n)}$. Then
\begin{equation}\label{chapter4, 4.2}
\frac{\sqrt{r_nh(r_n)}}{2} \le \left(\sqrt{r_nh(r_n)}-h(r_n) \right)= \delta_{\tilde{A}_n}(z_n)\le \delta_{\Omega}(z_n) \le \sqrt{r_n h(r_n)}.
\end{equation}
Consider the conformal mapping
$$
T: \tilde{A}_n \to A_{R_n}, \quad z \mapsto \frac{z-a_n}{\sqrt{r_nh(r_n)}},
$$
where $R_n=\sqrt{\frac{r_n}{h(r_n)}}$ and $A_{R_n}=\{ z\in \mathbb{C}: \frac{1}{R_n}<|z|<R_n  \}$. Then combining (\ref{chapter3, 3.6}) and (\ref{chapter4, 4.2}), we have
\begin{eqnarray}
\notag K^{(1)}_{\Omega}(z_n)  &\le& K^{(1)}_{\tilde{A}_n}(z_n)\\
\notag &=& K^{(1)}_{A_{R_n}}(T(z_n))|T'(z_n)|^4 \\
\notag &=& O\left( \frac{h(r_n)}{r_n} \cdot \frac{1}{r_n^2h(r_n)^2}   \right)\\
&=& O\left( \frac{1}{r_n^3h(r_n)} \right). \label{chapter4, 4.3}
\end{eqnarray}

$(1)$ If $h(t)=Ct^{\alpha}$ with $\alpha>1$, then
$$
K^{(1)}_{\Omega}(z_n)= O \left( \frac{1}{Cr_n^{\alpha+3}}  \right).
$$
In this case $\delta_{\Omega}(z_n) \asymp \sqrt{r_n h(r_n)}= C^{\frac{1}{2}}r_n^{\frac{\alpha+1}{2}}$. Substituting this into (\ref{chapter4, 4.3}) yields
$$
K^{(1)}_{\Omega}(z_n) = O\left( C^{\frac{\alpha+3}{\alpha+1}-1} \delta_{\Omega}(z_n)^{-\frac{2}{\alpha+1}\cdot (\alpha+3)}   \right)=O\left( C^{\frac{2}{\alpha+1}}\delta_{\Omega}(z_n)^{ -2-\frac{4}{\alpha+1}  } \right).
$$

$(2)$ If $h(t)=Ct(\log \frac{1}{t})^{-\beta}$ with $\beta>0$, then
$$
K^{(1)}_{\Omega}(z_n)=O\left( C^{-1}r_n^{-4}{\left(\log \frac{1}{r_n}\right)^{\beta}}  \right).
$$
Note that
$$
\frac{\sqrt{r_nh(r_n)}}{2}\le \delta_{\Omega}(z_n)\le \sqrt{r_nh(r_n)}=\sqrt{C}r_n\left(\log \frac{1}{r_n} \right)^{-\frac{\beta}{2}},
$$
so for $n$ sufficiently large we have
\begin{equation} \label{chapter4, 4.4}
C^2\delta_{\Omega}(z_n)^{-4} \asymp r_n^{-4}\left( \log \frac{1}{r_n} \right)^{2\beta},\,\,\,\log \frac{1}{\delta_{\Omega}(z_n)}\asymp \log \frac{1}{r_n}.
\end{equation}
Combining (\ref{chapter4, 4.3}) and (\ref{chapter4, 4.4}), we obtain
$$
K^{(1)}_{\Omega}(z_n)=O\left(C \delta_{\Omega}(z_n)^{-4} \left( \log \frac{1}{\delta_{\Omega}(z_n)} \right)^{-\beta}  \right).
$$
The constants implicit in the above $O$-notation are absolute numerical constants.
\end{proof}

\begin{proof}[Proof of Corollary \ref{corollary}]
We only prove $(1)$; the proof of $(2)$ is analogous. By the hypothesis, there exists a constant $M>0$ such that
\begin{equation}\label{chapter4, 4.5}
K_{\Omega}^{(1)}(w) \ge M \delta_{\Omega}(w)^{-2-\frac{4}{\alpha+1}}.
\end{equation}
On the other hand, suppose that $\partial \Omega$ does not satisfy condition $(U)_{1,\alpha}$. Then for every $C \in (0,1)$ and every $r_0>0$, there exist $a \in \partial \Omega$ and $r<r_0$ such that
$$
\{ z \in \mathbb{C}:  Ch_{1,\alpha}(r) < |z-a|<r   \} \cap \partial \Omega = \emptyset.
$$
As long as $r$ is chosen sufficiently small, we necessarily have 
$$\{ z \in \mathbb{C}:  Ch_{1,\alpha}(r) < |z-a|<r   \} \subset \Omega.$$ 
In particular, setting $r_0=\frac{1}{n}$, we obtain
$$
A_n=\{ z \in \mathbb{C}:  Ch_{1,\alpha}(r_n) < |z-a_n|<r_n   \} \subset \Omega.
$$
Therefore, by Theorem 1.3 together with (\ref{chapter4, 4.5}), there exists a sequence $\{ w_n \}\subset \Omega$ such that
$$
M\delta_{\Omega}(w_n)^{-2-\frac{4}{\alpha+1}} \le K_{\Omega}^{(1)}(w_n) \le C_0\cdot C^{\frac{2}{\alpha+1}}\delta_{\Omega}(w_n)^{-2-\frac{4}{\alpha+1}}.
$$
That is, for every $C \in (0,1)$,
$$
M \le C_0\cdot C^{\frac{2}{\alpha+1}},
$$
which is impossible.
\end{proof}

\section{Proof of Theorem \ref{lowerboundK1}}

Let $E\subset\mathbb{C}$ be a non-polar compact subset and let $\mu_E$ be its equilibrium measure. Following Zwonek \cite{Zwonek} and Pflug-Zwonek \cite{PflugZwonek}, we set
\begin{equation}
f_E(w):=\int_E\frac{d\mu_E(\zeta)}{w-\zeta}.
\end{equation}
It follows that $f_E$ is a holomorphic function on $\mathbb{C}\setminus{E}$. A simple dilatation of Lemma 2 in \cite{PflugZwonek} gives the following.

\begin{lemma}[cf. \cite{PflugZwonek, XiongZheng}]\label{lemma5.1}
If $E$ is a non-polar compact subset in $D(0,r)$, then 
$$
\int_{D(0,r)\setminus{E}}|f_E|^2 \lesssim \log \frac{4r}{\mathrm{Cap}(E)}.
$$
\end{lemma}

\begin{lemma}\label{lemma5.2}
Let $\Omega \subset \mathbb{C}$ be a domain, $w \in \Omega$, and $a \in \partial \Omega$. Set $E_n=\overline{D(a,r_n)}-\Omega$, where $r_n \le \frac{|w-a|}{n}$ for $n \in \mathbb{N}$. Let 
$$
f_n(z)=\int_{E_n}\frac{1}{z-\zeta} \mathrm{d}\mu_{E_n}(\zeta),
$$ 
where $\mu_{E_n}$ is the equilibrium measure of the compact set $E_n$. Then, as $n \to \infty$,
\begin{eqnarray*}
f_n(w)&=& \frac{1}{w-a}+O\left( \frac{1}{n|w-a|} \right),\\
f_n'(w)&=& -\frac{1}{(w-a)^2}+O\left( \frac{1}{n|w-a|^2} \right).
\end{eqnarray*}
The constants in the $O$-notation are absolute numerical constants.
\end{lemma}
\begin{proof}
Assume without loss of generality that $a=0$. We compute
\begin{eqnarray}\label{chapter5, 5.1}
\notag f_n(w)&=&\int_{E_n} \frac{1}{w-\zeta} \mathrm{d}\mu_{E_n}(\zeta)\\
\notag &=&\int_{E_n} \frac{1}{|w-\zeta|e^{i\mathrm{arg}(w-\zeta)}} \mathrm{d}\mu_{E_n}(\zeta)\\
\notag &=& \frac{1}{e^{i\mathrm{arg}(w)}} \int_{E_n} \frac{1}{|w-\zeta|}{e^{i(\mathrm{arg}(w)-\mathrm{arg}(w-\zeta))}} \mathrm{d}\mu_{E_n}(\zeta)\\
\notag &=& \frac{1}{e^{i\mathrm{arg}(w)}}  \int_{E_n} \frac{1}{|w-\zeta|}{\cos{\left(\mathrm{arg}(w)-\mathrm{arg}(w-\zeta)\right)}} \mathrm{d}\mu_{E_n}(\zeta) \\
&&+ \frac{1}{e^{i\mathrm{arg}(w)}}  \int_{E_n} \frac{1}{|w-\zeta|}{i\sin{\left(\mathrm{arg}(w)-\mathrm{arg}(w-\zeta)\right)}} \mathrm{d}\mu_{E_n}(\zeta).
\end{eqnarray}
Since $r_n\le \frac{|w|}{n}$, we have $\mathrm{arg}(w)-\mathrm{arg}(w-\zeta)\in \left(-\arcsin \frac{1}{n}, \arcsin \frac{1}{n}\right)$, and hence
\begin{eqnarray}
\cos \left( \mathrm{arg}(w)-\mathrm{arg}(w-\zeta)    \right)&=&1+O\left(\frac{1}{n}\right), \label{chapter5, 5.2} \\
\sin \left(\mathrm{arg}(w)-\mathrm{arg}(w-\zeta) \right)&=&O\left( \frac{1}{n}  \right).\label{chapter5, 5.3}
\end{eqnarray}
Moreover, since $\zeta \in E_n \subset \overline{D(0,r_n)}$, for $n$ sufficiently large we have
\begin{equation}\label{chapter5, 5.4}
\frac{1}{|w-\zeta|}-\frac{1}{|w|}\le \frac{|\zeta|}{|w||w-\zeta|} =O\left( \frac{1}{n|w|} \right).
\end{equation}
The constants in these $O$-estimates are absolute numerical constants. Substituting (\ref{chapter5, 5.1}), (\ref{chapter5, 5.2}), (\ref{chapter5, 5.3}), and (\ref{chapter5, 5.4}) into the above yields
\begin{eqnarray*}
f_n(w)-\frac{1}{w} &=&\frac{1}{e^{i\mathrm{arg}(w)}}  \int_{E_n} \left( \frac{1}{|w-\zeta|}-\frac{1}{|w|} \right) \mathrm{d}\mu_{E_n}(\zeta) +O\left(\frac{1}{n}\right)\int_{E_n} \frac{1}{|w-\zeta|} \mathrm{d}\mu_{E_n}(\zeta) \\
&=&O\left( \frac{1}{n|w|}  \right).
\end{eqnarray*}

Similarly, for
$$
f_n'(w)=\int_{E_n} \frac{-1}{(w-\zeta)^2} \mathrm{d}\mu_{E_n}(\zeta),
$$
we obtain
$$
f_n'(w)= -\frac{1}{w^2}+O\left( \frac{1}{n|w|^2} \right).
$$
\end{proof}

\begin{lemma}\label{lemma5.3}
Let $\Omega$ be a planar domain whose boundary contains three points $a_i$, $i=1,2,3$. Set $E_i=\overline{D(a_i,r)} - \Omega$, and define
$$
f_i(z)=\int_{E_i} \frac{1}{z-\zeta} \mathrm{d}\mu_{E_i}(\zeta), \,\,\,i=1,2,3.
$$
Fix a constant $s \in \mathbb{C}$, and set
$$
f(z)=f_1(z)-sf_1(z)-(1-s)f_2(z).
$$
If $R>0$ is sufficiently large so that $E_i \subset D(0,R)$ for $i=1,2,3$, then
$$
\int_{\Omega-\overline{D(0,2R)}} |f(z)|^2 \mathrm{d}\lambda(z) \le 16\pi(1+|s|)^2,
$$
where $\lambda$ denotes the usual Lebesgue measure.
\end{lemma}
\begin{proof}
We compute
\begin{eqnarray*}
f(z)&=&f_1(z)-sf_2(z)-(1-s)f_3(z)\\
&=&\int_{E_1} \frac{\mathrm{d}\mu_{E_1}(\zeta_1)}{z-\zeta_1}-s\int_{E_2} \frac{\mathrm{d}\mu_{E_2}(\zeta_2)}{z-\zeta_2}-(1-s)\int_{E_3} \frac{\mathrm{d}\mu_{E_3}(\zeta_3)}{z-\zeta_3}\\
&=&\int_{E_1}\int_{E_2}\int_{E_3} \left( \frac{1}{z-\zeta_1}-\frac{s}{z-\zeta_2}-\frac{1-s}{z-\zeta_3} \right) \mathrm{d}\mu_{E_1}(\zeta_1)\mathrm{d}\mu_{E_2}(\zeta_2)\mathrm{d}\mu_{E_3}(\zeta_3)\\
&=&\int_{E_1}\int_{E_2}\int_{E_3} \frac{(z-\zeta_2)(\zeta_1-\zeta_3)-s(z-\zeta_1)(\zeta_3-\zeta_2)}{(z-\zeta_1)(z-\zeta_2)(z-\zeta_3)} \mathrm{d}\mu_{E_1}(\zeta_1)\mathrm{d}\mu_{E_2}(\zeta_2)\mathrm{d}\mu_{E_3}(\zeta_3).
\end{eqnarray*}
where the third equality follows from the fact that the equilibrium measures are probability measures. Therefore,
\begin{equation}\label{chapter5, 5.5}
|f(z)|\le \int_{E_1}\int_{E_2}\int_{E_3} \left( \frac{|\zeta_1-\zeta_3|}{|z-\zeta_1||z-\zeta_3|}+|s|\frac{|\zeta_3-\zeta_2|}{|z-\zeta_2||z-\zeta_3|}   \right) \mathrm{d}\mu_{E_1}(\zeta_1)\mathrm{d}\mu_{E_2}(\zeta_2)\mathrm{d}\mu_{E_3}(\zeta_3).
\end{equation}
Note that
\begin{equation}\label{chapter5, 5.6}
 |z-\zeta_i| \ge |z|-R \ge  \frac{|z|}{2}, \,\,\,i=1,2,3.
\end{equation}
and
\begin{equation}\label{chapter5, 5.7}
|\zeta_1-\zeta_3| \le 2R,\,\,\, |\zeta_3-\zeta_2| \le 2R.
\end{equation}
Combining (\ref{chapter5, 5.5}), (\ref{chapter5, 5.6}), and (\ref{chapter5, 5.7}), we get
$$
|f(z)| \le \frac{8R(1+|s|)}{|z|^2}.
$$
Consequently,
\begin{eqnarray*}
\int_{\Omega-\overline{D(0,2R)}} |f(z)|^2 \mathrm{d}\lambda(z) &\le& \int_{\mathbb{C}-\overline{D(0,2R)}} \frac{64R^2(1+|s|)^2}{|z|^4} \mathrm{d}\lambda(z) \\
&=& (64R^2(1+|s|)^2) \int_{0}^{2\pi} \int_{2R}^{\infty} \frac{1}{t^4}\cdot t \mathrm{d}t \mathrm{d}\theta\\
&=& 16\pi(1+|s|)^2.
\end{eqnarray*}
\end{proof}

\begin{proposition} \label{prop5.4}
Let $\Omega \subset \mathbb{C}$ be a domain. For each point $w\in \Omega$, choose three boundary points $a_1, a_2, a_3 \in \partial \Omega$ such that $R:=|a_3-w|$ is much larger than $|a_2-w|$ and $|w-a_1|$, i.e.,
$$
\frac{R}{\max\{ |a_1-w|, |a_2-w| \} } \to +\infty,\,\,\,w \to \partial \Omega.
$$
Fix a sufficiently large constant $n$. Then there exists a numerical constant $C_0>0$ such that
\begin{equation}\label{chapter1, 1.1}
K^{(1)}_{\Omega}(w)\ge C_0 \frac{|a_2-a_1|^2}{(|w-a_1|^2|w-a_2|^4+|w-a_1|^4|w-a_2|^2)  \log \frac{8R}{\min_{i=1,2,3}\mathrm{Cap}(E_i)} },
\end{equation}
where $E_i=\overline{D\left(a_i, r \right)} \cap \Omega^c$ and $r=\frac{1}{n}\min \{ |w-a_1|,|w-a_2|, |a_2-a_1| \}$.
\end{proposition}

\begin{proof}
We may assume without loss of generality that $|w-a_1|\le |w-a_2|$. For $w \in \Omega$, define
$$
f_i(z)=\int_{E_i} \frac{1}{z-\zeta} \mathrm{d}\mu_{E_i}(\zeta).
$$
Let $f(z)=f_1(z)-sf_2(z)-(1-s)f_3(z)$, where the constant $s \in \mathbb{C}$ is chosen so that
$$
f(w)=f_1(w)-sf_2(w)-(1-s)f_3(w)=0.
$$
Then by Lemma \ref{lemma5.2}, together with the choice of $a_3$, when $w$ is sufficiently close to the boundary and $n$ is taken sufficiently large, we have
\begin{equation}\label{chapter5, 5.8}
|s|=\left| \frac{f_1(w)-f_3(w)}{f_2(w)-f_3(w)} \right| \asymp \left| \frac{f_1(w)}{f_2(w)}\right| \asymp \frac{|w-a_2|}{|w-a_1|}.
\end{equation}
As long as $R$ is much larger than $|a_2-a_1|$ and $|w-a_1|$, and $n$ is sufficiently large, the constants in the above ``$\asymp$'' are arbitrarily close to $1$.
Now
\begin{eqnarray}
\notag f'(w)&=&f'(w)+\frac{f(w)}{w-a_2}\\ 
\notag &=& \left( f_1'(w)+\frac{f_1(w)}{w-a_2} \right)-s\left( f_2'(w)+\frac{f_2(w)}{w-a_2}   \right) -(1-s)\left(  f_3'(w)+\frac{f_3(w)}{w-a_2}   \right)\\
&=:& I_1-sI_2-(1-s)I_3.\label{chapter5, 5.9}
\end{eqnarray}
First,
\begin{eqnarray}
\notag I_1&=&f_1'(w)+\frac{f_1(w)}{w-a_2}\\
\notag &=&\int_{E_1} \left( -\frac{1}{(w-\zeta)^2} +\frac{1}{(w-a_2)(w-\zeta)} \right) \mathrm{d}\mu_{E_1}(\zeta)\\
\notag &=&\int_{E_1} \frac{a_2-\zeta}{(w-\zeta)^2(w-a_2)} \mathrm{d}\mu_{E_1}(\zeta)\\
\notag &=& -\frac{a_2-a_1}{w-a_2}f_1'(w)-\int_{E_1} \frac{a_1-\zeta}{(w-\zeta)^2(w-a_2)}\mathrm{d}\mu_{E_1}(\zeta)\\
&=&\frac{a_2-a_1}{(w-a_2)(w-a_1)^2}\left( 1+O\left( \frac{1}{n} \right) \right)+O \left( \frac{|a_2-a_1|}{n|w-a_1|^2|w-a_2|}    \right),\label{chapter5, 5.10}
\end{eqnarray}
so $|I_1| \asymp \frac{|a_2-a_1|}{|w-a_1|^2|w-a_2|}$. For the remaining two terms, we have
\begin{eqnarray}
\notag |sI_2| &=& \left| s \int_{E_2} \frac{a_2-\zeta}{(w-\zeta)^2(w-a_2)} \mathrm{d}\mu_{E_2}(\zeta)   \right|\\
\notag &\le& |s|   \int_{E_2} \left| \frac{a_2-\zeta}{(w-\zeta)^2(w-a_2)}\right| \mathrm{d}\mu_{E_2}(\zeta)   \\
\notag &\lesssim& \frac{|w-a_2|}{|w-a_1|} \cdot \frac{\min \{ |a_2-a_1|,|w-a_1|,|w-a_2| \}}{n|w-a_2|^3}\\
\notag &\lesssim&\frac{|a_2-a_1|}{n|w-a_1||w-a_2|^2}\\
&\lesssim& \frac{|I_1|}{n},\label{chapter5, 5.11}
\end{eqnarray}
and
\begin{eqnarray}
\notag |(1-s)I_3| &=& \left|(1- s) \int_{E_3} \frac{a_2-\zeta}{(w-\zeta)^2(w-a_2)} \mathrm{d}\mu_{E_3}(\zeta)   \right|\\
\notag &\lesssim& \left( 1+\frac{|w-a_2|}{|w-a_1|} \right) \cdot \frac{ |a_2-a_3| }{|w-a_2||w-a_3|^2}\\
&\ll &| I_1|, \label{chapter5, 5.12}
\end{eqnarray}
where the last inequality follows from the fact that $R=|w-a_3|$ is much larger than $|w-a_2|$ and $|w-a_1|$. Combining (\ref{chapter5, 5.9}), (\ref{chapter5, 5.10}), (\ref{chapter5, 5.11}), and (\ref{chapter5, 5.12}), we see that as long as $n$ is sufficiently large and $R$ is much larger than $|a_2-a_1|$ and $|w-a_1|$, we have
\begin{equation}\label{chapter5, 5.13}
|f'(w)| \gtrsim \frac{|a_2-a_1|}{|w-a_1|^2|w-a_2|}.
\end{equation}
Here, the constant in "$\gtrsim$" is a numerical constant.

We now estimate $\|f\|_{L^2(\Omega)}$. First,
\begin{eqnarray} \label{chapter5, 5.14}
\int_{\Omega} |f|^2 \mathrm{d}\lambda &=& \int_{\Omega\cap D(0,2R)} |f|^2 \mathrm{d}\lambda +\int_{\Omega-D(0,2R)} |f|^2 \mathrm{d}\lambda.
\end{eqnarray}
For the second term, since $R$ is taken sufficiently large, we have $E_i \subset D(0, 2R)$ for $i=1,2,3$. Hence, by Lemma \ref{lemma5.3},
\begin{equation} \label{chapter5, 5.15}
\int_{\Omega-D(0,2R)} |f|^2 \mathrm{d}\lambda \le 16\pi (1+|s|)^2.
\end{equation}
For the first term, using Lemma \ref{lemma5.1}, we obtain
\begin{eqnarray}
\notag \int_{\Omega\cap D(0,2R)} |f|^2 \mathrm{d}\lambda &\lesssim& \int_{D(0,2R)} (|f_1|^2+|s|^2|f_2|^2+|1-s|^2|f_3|^2) \mathrm{d}\lambda \\
\notag &\lesssim& (1+|s|^2) \left(  \log \frac{8R}{\mathrm{Cap}(E_1)} +\log \frac{8R}{\mathrm{Cap}(E_2)}+\log \frac{8R}{\mathrm{Cap}(E_3)}    \right)\\
&\lesssim&(1+|s|^2) \log \frac{8R}{\min_{i=1,2,3}\mathrm{Cap}(E_i)}. \label{chapter5, 5.16}
\end{eqnarray}
Here, the constant in "$\gtrsim$" is also a numerical constant. Combining (\ref{chapter5, 5.13}), (\ref{chapter5, 5.14}), (\ref{chapter5, 5.15}), and (\ref{chapter5, 5.16}), we get
$$
K^{(1)}_{\Omega}(w)\gtrsim \frac{|a_2-a_1|^2}{|w-a_1|^4|w-a_2|^2(1+|s|^2)\log \frac{8R}{\min_{i=1,2,3}\mathrm{Cap}(E_i)} }.
$$
Since $|s|^2 \asymp \frac{|w-a_2|^2}{|w-a_1|^2}$, it follows that
$$
K^{(1)}_{\Omega}(w)\gtrsim \frac{|a_2-a_1|^2}{(|w-a_1|^2|w-a_2|^4+|w-a_1|^4|w-a_2|^2)  \log \frac{8R}{\min_{i=1,2,3}\mathrm{Cap}(E_i)} }.
$$
This proves (\ref{chapter1, 1.1}).
\end{proof}

\begin{proof}[Proof of Theorem \ref{lowerboundK1}]
For $w \in \Omega$, we may take $a_1$ to be a point on the boundary closest to $w$, and assume without loss of generality that $a_1=0$, so that $|w|=\delta_{\Omega}(w)$. Choose $r>0$ such that $|w|=h(r)$. By the $h$-uniform perfectness condition, we may take
$$
a_2 \in \partial \Omega \cap \{ z\in \mathbb{C}: h(r) \le |z| \le r   \},
$$
and
$$
a_3 \in \partial \Omega \cap \{ z\in \mathbb{C}: 2r \le |z| \le h^{-1}(2r)   \}.
$$
Note that all the functions $h$ considered in this theorem have inverse functions. Set $R=h^{-1}(2r)$. Note that $|w-a_2| \ge |w|$ and $|w| \le |a_2|$, i.e.,
$$
\max \{ |w|, |a_2-w|  \}  \le |w-a_2| \lesssim |a_2|.
$$
Discussing separately the two cases $|a_2|\ge 2|w|$ and $|w| \le |a_2| <2|w|$, it is easy to see that in either case we always have $|a_2| \asymp |w-a_2|$. Consequently, Proposition \ref{prop5.4} implies
\begin{eqnarray}\label{chapter5, 5.17}
\notag K^{(1)}_{\Omega}(w) &\gtrsim& \left( |w|^{2}|w-a_2|^{2}  \log \frac{8h^{-1}(2h^{-1}(|w|))}{\min_{i=1,2,3}\mathrm{Cap}(E_i)} \right)^{-1} \\
&\gtrsim& \left( r^{2}h(r)^{2} \log \frac{8h^{-1}(2h^{-1}(|w|))}{\min_{i=1,2,3}\mathrm{Cap}(E_i)}   \right)^{-1},
\end{eqnarray}
where $E_i=\overline{D\left(a_i, \frac{|w|}{n}\right)} \cap E^c, i=1,2,3$.

$(1)$ When $h(t)=Ct^{\alpha}$ with $\alpha \in (1,2)$, we have $h^{-1}(t)=\left( \frac{t}{C}  \right)^{\frac{1}{\alpha}}$, and by Theorem \ref{thm2.6},
$$
\mathrm{Cap}(E_i) \gtrsim \left(\frac{|w|}{n}\right)^{\frac{\alpha}{2-\alpha}}.
$$
Here, $n$ is a properly chosen constant in advance. Substituting this into (\ref{chapter5, 5.17}) yields
$$
K^{(1)}_{\Omega}(w)\gtrsim \frac{1}{r^{2+2\alpha}\log \frac{1}{r}} \gtrsim \frac{1}{|w|^{2+\frac{2}{\alpha}}\log \frac{1}{|w|}}.
$$

$(2)$ When $h(t)=Ct\left( \log \frac{1}{t} \right)^{-\beta}$ with $\beta>0$, from the proof of Lemma 5.2 in \cite{XiongZheng} it is easy to see that
$$
h^{-1}(t) \lesssim \frac{1}{C}t \left( \log \frac{1}{t} \right)^{\beta},
$$
and by Theorem \ref{thm2.6},
$$
\mathrm{Cap}(E_i) \gtrsim \frac{|w|}{n} \left(\log\frac{n}{|w|}\right)^{-2\beta}.
$$
Substituting this into (\ref{chapter5, 5.17}) gives
$$
K^{(1)}_{\Omega}(w) \gtrsim \frac{1}{|w|^4\left( \log \frac{1}{|w|}  \right)^{2\beta} \log \log \frac{1}{|w|}}.
$$
\end{proof}

\section{Proof of Theorem \ref{estimateb}}

In this section, we use the notation
$$
h^{-1}(r)=\sup \{ t: h(t) = r   \},  \ \ \ \forall r \in \mathrm{Im}(h),
$$
where the function $h(r)$ is continuous, but need not be strictly increasing. Clearly, $h(h^{-1}(r)) = r$. 

To prove Theorem \ref{estimateb}, we first need to establish a lower bound estimate for $K_{\Omega}(w)$. Here we reformulate Theorem 1.4 of \cite{XiongZheng} as follows.

\begin{proposition}\label{prop6.1}
Let $\Omega$ be a domain in $\mathbb{C}$, and let $w\in\Omega$ be sufficiently close to the boundary. If $\Omega$ satisfies conditions $(U)_{h}$ and $(C)_g$, then
$$
K_{\Omega}(w) \gtrsim \frac{1}{\delta_{\Omega}(w)^2 \log \frac{8h^{-1}(8\delta_{\Omega}(w))}{g(\delta_{\Omega}(w))}}.
$$
\end{proposition}

\begin{proof}
Let $w\in\Omega$ be sufficiently close to $\partial\Omega$, and choose $w'\in\partial\Omega$ such that $|w-w'|=\delta_\Omega(w)$. Take $r=h^{-1}(8\delta_{\Omega}(w))$. Then there exists another point $w''\in\partial\Omega$ with
\[
8\delta_\Omega(w)\leq|w''-w'|\leq{r}.
\]
Let $E_1:=\overline{D(w',\delta_\Omega(w))}\setminus\Omega$. As in \cite{PflugZwonek}, we divide $E_1$ into the following three parts:
\begin{eqnarray*}
E_{11} &=& E_1\cap\left\{w+se^{i\theta}\in\mathbb{C}: s>0,\ -\frac{\pi}{3}\leq\theta\leq\frac{\pi}{3}\right\},\\
E_{12} &=& E_1\cap\left\{w+se^{i\theta}\in\mathbb{C}: s>0,\ \frac{\pi}{3}\leq\theta\leq\pi\right\},\\
E_{13} &=& E_1\cap\left\{w+se^{i\theta}\in\mathbb{C}: s>0,\ \pi\leq\theta\leq\frac{5\pi}{3}\right\},
\end{eqnarray*}
so that
\begin{equation}\label{eq:cos}
\cos(\mathrm{arg}(\zeta-w))\geq\frac{1}{2},\ \ \ \zeta\in{E_{11}}.
\end{equation}
Moreover, by rotating $\Omega$ around $w$, we may assume that $\mathrm{Cap}(E_{11})\geq\mathrm{Cap}(E_{12})\geq\mathrm{Cap}(E_{13})$. Then we infer from Theorem 5.1.4 of \cite{Ransford} that
\begin{eqnarray*}
\frac{1}{\log ({8r}/{\mathrm{Cap}(E_1)})} &\le& \frac{1}{\log ({8r}/{\mathrm{Cap}(E_{11})})}+\frac{1}{\log ({8r}/{\mathrm{Cap}(E_{12})})}+\frac{1}{\log ({8r}/{\mathrm{Cap}(E_{13})})}\\
&\le& \frac{3}{\log ({8r}/{\mathrm{Cap}(E_{11})})},
\end{eqnarray*}
so that
\begin{equation}\label{eq:log_E_11}
\log \frac{8r}{\mathrm{Cap}(E_{11})} \leq 3\log \frac{8r}{\mathrm{Cap}(E_1)}.
\end{equation}

Set $E_2:=\overline{D(w'',\delta_\Omega(w))}\setminus\Omega$. We have $E_{11},E_2\subset{D(w,2r)}$. Let
\[
f_{11}:=f_{E_{11}},\ \ \ f_2:=f_{E_2},
\]
be the functions given in Lemma \ref{lemma5.1}, and consider $f:=f_{11}-f_2$. Write
\[
\int_{\Omega}|f|^2 = \int_{\Omega\cap D(w,2r)}|f|^2+\int_{\Omega\setminus{D(w,2r)}}|f|^2=:I_1+I_2.
\]
For $I_1$, it follows from Lemma \ref{lemma5.1} and \eqref{eq:log_E_11} that
\begin{eqnarray}
I_1 &\lesssim& \int_{\Omega\cap D(w,2r)}\left(|f_{11}|^2+|f_2|^2\right)\nonumber\\
&\lesssim& \int_{D(w,2r)\setminus{E_{11}}}|f_{11}|^2+\int_{D(w,2r)\setminus{E_2}}|f_2|^2\nonumber\\
&\lesssim& \log \frac{8r}{\mathrm{Cap}(E_{11})} + \log \frac{8r}{\mathrm{Cap}(E_2)}\nonumber\\
&\lesssim& \log \frac{8r}{\mathrm{Cap}(E_1)}+\log \frac{8r}{\mathrm{Cap}(E_2)}  \nonumber\\
&\lesssim& \log \frac{8r}{g(\delta_{\Omega}(w))},
\end{eqnarray}
where the last inequality follows from condition $(C)_g$, and the implicit constant is independent of $w$. For $I_2$, Lemma \ref{lemma5.3} implies
$$
I_2 \lesssim 1.
$$
Thus,
\begin{equation}\label{eq:f_L2_thm2}
\int_\Omega|f|^2\lesssim\log\frac{8h^{-1}(8\delta_{\Omega}(w))}{g(\delta_\Omega(w))}.
\end{equation}

It remains to find a lower bound for $|f(w)|$. Clearly, we have $|f(w)|\geq|f_{11}(w)|-|f_2(w)|$. First, by using \eqref{eq:cos} and noticing that $E_{11}\subset{D(w,2\delta_\Omega(w))}$, we have
\begin{eqnarray*}
|f_{11}(w)|&=&\left| \int_{E_{11}}\frac{d\mu_{K_{11}}(\zeta)}{w-\zeta}  \right|\\
&=& \left| \int_{E_{11}} \frac{\cos(\arg(\zeta-w))}{|w-\zeta|} d\mu_{K_{11}}(\zeta) +i \int_{E_{11}} \frac{\sin(\arg(\zeta-w))}{|w-\zeta|} d\mu_{K_{11}}(\zeta) \right| \\
&\ge&\left| \int_{E_{11}} \frac{\cos(\arg(\zeta-w))}{|w-\zeta|} d\mu_{E_{11}}(\zeta)  \right| \\
&\ge& \frac{1}{2}\int_{E_{11}}\frac{1}{|w-\zeta|}d\mu_{E_{11}}(\zeta)\\
&\ge& \frac{1}{4\delta_\Omega(w)}.
\end{eqnarray*}
On the other hand,
\begin{eqnarray*}
|f_2(w)| &\le& \int_{E_2} \frac{d\mu_{K_2}(\zeta)}{|w-\zeta|}\\
&\le& \int_{E_2} \frac{d\mu_{K_2}(\zeta)}{|w'-w''|-|w-w'|-|w''-\zeta|}\\
&\le& \int_{E_2} \frac{d\mu_{K_2}(\zeta)}{8\delta_\Omega(w)-\delta_\Omega(w)-\delta_\Omega(w)}\\
&\le& \frac{1}{6\delta_\Omega(w)}.
\end{eqnarray*}
Thus $|f(w)|\geq \frac{1}{12\delta_\Omega(w)}$, which together with \eqref{eq:f_L2_thm2} completes the proof of the proposition.
\end{proof}

Now we apply the above proposition to the proof of Theorem \ref{estimateb}.

\begin{proof}[Proof of Theorem \ref{estimateb}]
Write $E=\partial \Omega$, and $s_i(t):=\frac{h_i(t)}{t}, i=1,2$. As noted in Section 2, we may assume without loss of generality that $g_E(t)>0$ for all $t \in (0, \mathrm{diam}(\Omega))$.
For $w \in \Omega$, Proposition \ref{prop6.1} gives
$$
K_{\Omega}(w) \gtrsim \frac{1}{\delta_{\Omega}(w)^2 \log \frac{8h_{1}^{-1}(8\delta_{\Omega}(w))}{g_E(\delta_{\Omega}(w))} }.
$$

On the other hand, there exists a sequence $r_n$ monotonically decreasing to $0$ such that for each $r_n$, there is some $a_n \in E$ with
$$
A_n=\{  z \in \mathbb{C}: h_2(r_n)<|z-a_n|<r_n    \} \subset E^c.
$$
For $n$ sufficiently large, we necessarily have $A_n\subset \Omega$; hence we may assume this holds for every $n$. Take $w_n \in A_n$ such that $|w_n-a_n|^2=r_nh_2(r_n)$. We estimate the upper bound of $K_{\Omega}^{(1)}(w_n)$. Consider the conformal mapping
$$
T: A_n \to \left\{ z:    \sqrt{\frac{h_2(r_n)}{r_n}}<|z|<\sqrt{\frac{r_n}{h_2(r_n)}}  \right\}, \ \ \ w \mapsto \frac{w-a}{\sqrt{r_n h_2(r_n)}}.
$$
Note that
$$
h_1(r_n)<h_2(r_n)\ll \sqrt{r_n h_2(r_n)}\lesssim \sqrt{r_n h_2(r_n)}-h_2(r_n) \le \delta_{\Omega}(w_n) \le \sqrt{r_n h_2(r_n)}\ll r_n,
$$
that is,
$$
\delta_{\Omega}(w_n) \asymp \sqrt{r_n h_2(r_n)}.
$$
Then $|T(w_n)|=1$. Consequently, by the properties of $K^{(1)}_{\Omega}(w)$ and (3.6), we have
\begin{eqnarray*}
K_{\Omega}^{(1)}(w_n) &\le& K_{A_n}^{(1)}(w_n)\\
&=& K_{T(\Omega)}^{(1)}(T(w_n)) \cdot | T'(w_n) |^4 \\
&=&O\left( \frac{h_2(r_n)}{r_n} \right) \cdot \frac{1}{r_n^2h_2(r_n)^2}\\
&=&\frac{1}{\delta_{\Omega}(w_n)^4}\cdot O\left(  \frac{h_2(r_n)}{r_n}    \right).
\end{eqnarray*}
Therefore,
\begin{eqnarray*}
b_{\Omega}(w_n)^2 &=&\frac{K_{\Omega}^{(1)}(w_n)}{K_{\Omega}(w_n)}\\
&\lesssim& \frac{1}{\delta_{\Omega}(w_n)^2} \cdot \frac{h_2(r_n)}{r_n} \cdot \log \frac{8h_{1}^{-1}(8\delta_{\Omega}(w_n))}{g_E(\delta_{\Omega}(w_n))}\\
&\lesssim& \frac{1}{\delta_{\Omega}(w_n)^2} \cdot \frac{h_2(r_n)}{r_n} \cdot \left(   \log \frac{\tilde{h}_{1}^{-1}(r_n)}{g_E(2\tilde{h}_1(r_n))}  \right),
\end{eqnarray*}
where $\tilde{h}(t)=\frac{1}{5}h_1(t)$. Note that by Theorem \ref{thm2.6},
$$
g_E(t) \ge \exp \left( \sum_{k=1}^{\infty} \frac{\log \tilde{h}_{1}^{\circ k}\left(\frac{t}{2}\right)}{2^k}   \right),
$$
which implies
\begin{eqnarray*}
b_{\Omega}(w_n)^2 &\lesssim&  \delta_{\Omega}(w_n)^{-2}(I_1+I_2+I_3),
\end{eqnarray*}
where
\begin{eqnarray*}
I_1&:=&  {{s}_2(r_n)}    \log \frac{ \tilde{h}_{1}^{-1}(r_n)}{r_n},\\
I_2&:=&{{s}_2(r_n)}  \log \frac{r_n}{\tilde{h}_{1}(r_n)},\\
I_3&:=&{{s}_2(r_n)}  \left( \log \tilde{h}_{1}(r_n) -\sum_{k=1}^{\infty} \frac{\log \tilde{h}_{1}^{\circ (k+1)}(r_n)}{2^k}    \right).
\end{eqnarray*}
This inequality also holds trivially when the series in $I_3$ diverges.

By the hypothesis,
$$
 \log \frac{ \tilde{h}_{1}^{-1}(r_n)}{r_n} \le \log \frac{ \tilde{h}_{1}^{-1}(r_n)}{\tilde{h}_{1}(\tilde{h}_{1}^{-1}(r_n))} \lesssim  \log \frac{1}{s_1(r_n)} \asymp \log \frac{1}{s_2(r_n)}.
$$
Hence $I_1, I_2 \to 0$ as $n \to \infty$. Note that
\begin{eqnarray*}
&&\log \tilde{h}_{1}(r_n) -\sum_{k=1}^{\infty} \frac{\log \tilde{h}_{1}^{\circ (k+1)}(r_n)}{2^k}\\
&=&\log  \tilde{h}_{1}(r_n)-\log  \tilde{h}_{1}^{\circ 2}(r_n)+\frac{1}{2} \log  \tilde{h}_{1}^{\circ 2}(r_n) -  \sum_{k=2}^{\infty} \frac{\log \tilde{h}_{1}^{\circ (k+1)}(r_n)}{2^k}\\
&=& \log \frac{ \tilde{h}_{1}(r_n)}{ \tilde{h}_{1}^{\circ 2}(r_n)} +\frac{1}{2} \log \frac{ \tilde{h}_{1}^{\circ 2}(r_n)}{ \tilde{h}_{1}^{\circ 3}(r_n)}+\frac{1}{2^2} \log  \tilde{h}_{1}^{\circ 3}(r_n)-\sum_{k=3}^{\infty} \frac{\log \tilde{h}_{1}^{\circ (k+1)}(r_n)}{2^k}\\
&=& \cdots \\
&=& \sum_{k=1}^{\infty} \frac{1}{2^{k-1}} \log \frac{ \tilde{h}_{1}^{\circ k}(r_n)}{ \tilde{h}_{1}^{\circ (k+1)}(r_n)}\\
&=& \sum_{k=1}^{\infty} \frac{1}{2^{k-1}} \log \frac{1}{ s_1(\tilde{h}_{1}^{\circ k}(r_n))}.
\end{eqnarray*}
To complete the proof, we need
$$
I_3=\sum_{k=1}^{\infty} \frac{1}{2^{k-1}} s_2(r_n) \log \frac{1}{s_1( \tilde{h}_{1}^{\circ k}(r_n))} \to 0, \ \ \ n \to +\infty.
$$
By the hypothesis,
$$
\log \frac{\tilde{h}_{1}\circ \tilde{h}_{1}(t)}{\tilde{h}_{1}(t)} \ge C \log \frac{\tilde{h}_{1}(t)}{t}, \ \ \ C<2.
$$
Then
$$
\log \frac{\tilde{h}_{1}^{\circ (k+1)}(t)}{\tilde{h}_{1}^{\circ k}(t)} \ge C \log \frac{\tilde{h}_{1}^{\circ k}(t)}{\tilde{h}_{1}^{\circ (k-1)}(t)} \ge \cdots \ge C^k  \log \frac{\tilde{h}_{1}(t)}{t},
$$
that is,
\begin{eqnarray*}
\sum_{k=1}^{\infty} \frac{1}{2^{k-1}} s_2(r_n) \log \frac{1}{s_1( \tilde{h}^{\circ k}(r_n))} &\le& \sum_{k=1}^{\infty} \left(\frac{C}{2} \right)^{k-1} s_2(r_n)\log \frac{1}{s_1(r_n)}\\
&\asymp&s_2(r_n)\log \frac{1}{s_2(r_n)}.
\end{eqnarray*}
Since $s_2(r_n) \to 0$, we have $I_3 \to 0$. This completes the proof.
\end{proof}

\section{Proofs of Theorems \ref{upperboundK}  and \ref{lowerboundd} }

The method of proof of Theorem \ref{upperboundK} is inspired by \cite{Chen2023}. Before starting the proof, we need to make some preparations.

We shall need some properties of Dirichlet capacity and Green capacity, these can be found in Section 2 of \cite{Chen2023}. 

Let $\Omega$ be a domain containing a compact set $E$, and let $g_{\Omega}(z,w)$ denote the (negative) Green function of $\Omega$. If in the definition of logarithmic capacity one replaces the energy $I(\mu)$ by
$$
I_G(\mu)=\int_{\Omega} \int_{\Omega} g_{\Omega}(z,w) \mathrm{d}\mu(z) \mathrm{d}\mu(w),
$$
then, in analogy with the definition of logarithmic capacity, one can define the Green capacity of $E$ with respect to $\Omega$ as
$$
C_g(E, \Omega):= \exp \left({\max_{\mu \in \mathcal{P}(E)}I_G(\mu)}\right),
$$
Where $\mathcal{P}(E)$ denotes the set of all Borel probability measures on $E$. The Green capacity is related to logarithmic capacity by
\begin{equation}
\log \frac{\mathrm{Cap}(E)}{R} \le \log C_g(E,\Omega) \le \log \frac{\mathrm{Cap}(E)}{d},
\end{equation}
where $R=\mathrm{diam} (\Omega)$ and $d=d(E,\partial \Omega)$.

The Dirichlet capacity of $E$ with respect to $\Omega$ is defined by
\begin{equation}\label{chapter2, 2.1}
C_d(E,\Omega)=\inf_{\phi \in \mathcal{L}(E,\Omega)} \int_{\Omega}|\nabla \phi|^2,
\end{equation}
where $\mathcal{L}(E,\Omega)$ denotes the class of all locally Lipschitz functions $\phi$ on $\Omega$ such that $0 \le \phi \le 1$ and $\phi|_E=1$. By the Dirichlet principle, the function $\phi_{min}$ attaining the infimum in (\ref{chapter2, 2.1}) is precisely the Perron solution of the following Dirichlet problem on $\Omega-E$:
$$
\Delta u=0; \quad u=0\ \text{n.e. on}\ \partial \Omega; \quad u=1\ \text{n.e. on}\ \partial E.
$$
Dirichlet capacity is conformally invariant, and its relation with Green capacity is given by
\begin{equation}
\frac{C_d(E,\Omega)}{2\pi} =-\frac{1}{\log C_g(E, \Omega)}.\label{chapter2, 2.2}
\end{equation}
We shall need the following property.

\begin{proposition}[cf. \cite{Chen2023, Grigoryan1999}] \label{prop7.1}
Let $\Omega \subset \mathbb{C}$ be a bounded domain, and let $U$ be a relatively compact open subset of $\Omega$. Then for any $w \in U$,
\begin{equation}\label{chapter2, 2.3}
\min_{\partial U}(-g_{\Omega}(\cdot, w)) \le \frac{2\pi}{C_d(E, \Omega)} \le \max_{\partial U} (-g_{\Omega}(\cdot,w)).
\end{equation}
\end{proposition}

We also need the following result of B{\l}ocki--Zwonek \cite{BlockiZwonek2018}.

\begin{theorem}[\cite{BlockiZwonek2018}]\label{thm7.2}
Let $\Omega \subset \mathbb{C}$ be a domain, $w \in \Omega$, and $r \in (0, \delta_{\Omega}(w))$. Then
$$
K_{\Omega}(w) \le \frac{1}{2\pi r^2 \left( -\max_{\partial D(w,r)}g_{\Omega}(\cdot, w)   \right)}.
$$
\end{theorem}

\begin{proof}[Proof of Theorem \ref{upperboundK}]
Let $r=\frac{\delta_{\Omega}(w)}{3}$. First, we reduce the problem to estimating 
$$C_d\left(\overline{D\left(w,\frac{\delta_{\Omega}(w)}{3}\right)},\Omega \right).$$ 
Set
$$
h_w(z):=\log \frac{|z-w|}{r}-g_{\Omega}(z,w).
$$
Then $h_w$ is harmonic on $\Omega$, equals $-g_{\Omega}(z,w)$ on $\partial D(w,r)$, and is nonnegative on $D(w, \delta(w))$. Therefore, by Harnack's inequality,
$$
\max_{z\in \partial D(w,r)} h_w(z) \le \frac{\delta_{\Omega}(w)+r}{\delta_{\Omega}(w)-r}h_w(0) \le \left( \frac{\delta_{\Omega}(w)+r}{\delta_{\Omega}(w)-r} \right)^2 \min_{z\in \partial D(w,r)}h_w(z).
$$
That is,
$$
\max_{\partial D(w,r)} (-g_{\Omega}(\cdot, w)) \le  \left( \frac{\delta_{\Omega}(w)+r}{\delta_{\Omega}(w)-r} \right)^2 \min_{\partial D(w,r)}(-g_{\Omega}(\cdot, w)).
$$
Combining Theorem \ref{thm7.2} with (\ref{chapter2, 2.3}), we obtain
\begin{eqnarray}
\notag K_{\Omega}(w)&\le & \frac{1}{2\pi r^2} \cdot \frac{1}{\min_{\partial D(w,r)}(-g_{\Omega}(\cdot,w))}\\
\notag &\le & \frac{1}{2\pi r^2}  \left( \frac{\delta_{\Omega}(w)+r}{\delta_{\Omega}(w)-r} \right)^2 \cdot \frac{C_d\left(\overline{D(w,r)},\Omega\right)}{2\pi}\\
&=& \frac{9}{\pi^2 \delta_{\Omega}(w)^2}\cdot C_d\left(\overline{D\left(w,\frac{\delta_{\Omega}(w)}{3}\right)},\Omega\right).\label{3.1}
\end{eqnarray}

Next, by the definition of Dirichlet capacity,
$$
C_d\left(\overline{D\left(w,\frac{\delta_{\Omega}(w)}{3}\right)},\Omega\right)=C_d \left( \Omega^c, \mathbb{C}_{\infty} \setminus \overline{D\left(w, \frac{\delta_{\Omega}(w)}{3}\right)} \right).
$$
Consider the mapping
$$
T:  \mathbb{C}_{\infty} \setminus \overline{D\left(w, \frac{\delta_{\Omega}(w)}{3}\right)} \to \mathbb{D}, \quad z \mapsto \frac{\delta_{\Omega}(w)}{3(z-w)}.
$$
Then $T(\Omega^c)\subset D\left(0,\frac{1}{3}\right)$. By conformal invariance of Dirichlet capacity,
$$
C_d \left( \Omega^c, \mathbb{C}_{\infty} \setminus \overline{D\left(w, \frac{\delta_{\Omega}(w)}{3}\right)} \right)=C_d(T(\Omega^c),\mathbb{D}).
$$
Let $E(w,R)=\overline{D(w,R)} \setminus \Omega$, with $R$ to be determined. Since $$\Omega^c=E(w,R)\cup \left( \overline{\Omega^c \setminus E(w,R)}   \right),$$ we have
$$
T(\Omega^c)=T(E(w,R))\cup T\left(\overline{\Omega^c \setminus E(w,R)}  \right).
$$
Let $d_1=d(T(\Omega^c),\mathbb{D}) \ge \frac{2}{3}$. Then by (\ref{chapter2, 2.2}),
\begin{eqnarray}
\notag \frac{C_d(T(\Omega^c),\mathbb{D})}{2\pi}&=&\frac{2\pi}{-\log C_g(T(\Omega^c), \mathbb{D})}\\
\notag&\le&\frac{1}{-(\log \mathrm{Cap}(T(\Omega^c))-\log d_1)}\\
\notag&\le&\frac{1}{{\log \frac{2}{3\mathrm{Cap}(T(\Omega^c))}}}\\
&\le& \frac{1}{{\log \frac{2}{3\mathrm{Cap}(T(E(w,R)))}}} +\frac{1}{\log \frac{2}{3\mathrm{Cap}(T(\overline{\Omega^c-E(w,R)}))}}        , \label{3.2}
\end{eqnarray}
where the last inequality follows from Theorem 5.1.4 of \cite{Ransford}. Since
$$
T(\overline{\Omega^c-T(E(w,R))})\subset D\left(0,\frac{\delta_{\Omega}(w)}{3R}\right),
$$
the second term is bounded by $\left( \log \frac{2R}{\delta_{\Omega}(w)}  \right)^{-1}$. On the other hand, the mapping $T$ satisfies
$$
\left| T(z_1)-T(z_2)    \right|=\frac{\delta_{\Omega}(w)}{3} \cdot \frac{|z_1-z_2|}{|z_1-w||z_2-w|}\le \frac{|z_1-z_2|}{3\delta_{\Omega}(w)}, \quad \forall z_1, z_2 \in E(w,R).
$$
By Theorem 2.1 (2),
$$
\mathrm{Cap}(T(E(w,R)))\le \frac{1}{3\delta_{\Omega}(w)}\cdot \mathrm{Cap}(E(w,R)).
$$
Therefore, when $\mathrm{Cap}(E(w,R)) \le 2\delta_{\Omega}(w)$, it follows from (\ref{3.2}) that
$$
C_d(T(\Omega^c),\mathbb{D})\le \frac{2\pi}{   \log \frac{2\delta_{\Omega}(w)}{\mathrm{Cap}(E(w,R))} }+\frac{2\pi}{ \log \frac{2R}{\delta_{\Omega}(w)}       }.
$$
Substituting this estimate into (\ref{3.1}) completes the proof.
\end{proof}
\begin{proof}[Proof of Theorem \ref{lowerboundd}]
$(1)$ Let $h(t)=Ct^{\alpha}$. On each $A_k$, when
$$w\in  \{ z\in \mathbb{C}: 2h(r_k) \le |w| \le 3h(r_k) \},$$
we estimate $K^{(1)}_{\Omega}(w)$ and $K_{\Omega}(w)$ separately.

First, take $a_1=0$, $a_2 \in \partial D(0,h(r_k))$, and $a_3 \in \partial D(0,r_k)\cap \partial \Omega$. This is possible because $\partial \Omega$ is assumed to be $h$-uniformly perfect; otherwise, if the inner or outer boundary of $A_k$ had no intersection with $\partial \Omega$, a contradiction would arise. Let $E_i=\overline{D\left(a_i, \frac{h(r_k)}{n}\right)}-\Omega$, where $n$ is a fixed sufficiently large integer. Then by Proposition \ref{prop5.4} and Theorem \ref{thm2.6},
\begin{eqnarray*}
K^{(1)}_{\Omega}(w) &\gtrsim& \frac{|a_2|^2}{(|w|^2|w-a_2|^4+|w|^4|w-a_2|^2) \log \frac{8r_k}{\min_{i=1,2,3} \mathrm{Cap}(E_i)}}\\
&\gtrsim& \frac{1}{h(r_k)^4\log \frac{8r_k}{\left(\frac{h(r_k)}{n}\right)^{\frac{\alpha}{2-\alpha}}}}\\
&\gtrsim& \frac{1}{h(r_k)^4 \log \frac{1}{r_k}}.
\end{eqnarray*}

By Theorem \ref{upperboundK},
$$
K_{\Omega}(w) \lesssim \frac{1}{\delta_{\Omega}(w)^2} \left( \frac{1}{\log \frac{2\delta_{\Omega}(w)}{\mathrm{Cap}(E(w,R))}} +\frac{1}{\log \frac{2R}{\delta_{\Omega}(w)}}   \right).
$$
Take $R=\frac{r_k}{2}$. Note that $\delta_{\Omega}(w) \asymp h(r_k)$, and by the hypothesis,
$$\mathrm{Cap}(E(w, R))=\mathrm{Cap}(\overline{D(0, h(r_k)) \setminus \Omega}) \lesssim h(r_k)^{\alpha'}.$$
Therefore,
$$
K_{\Omega}(w) \lesssim \frac{1}{h(r_k)^2 \log \frac{1}{r_k}}.
$$

Combining the above estimates, at the point $w$ we have
$$b_{\Omega}(w) \gtrsim \frac{1}{h(r_k)} \asymp \frac{1}{|w|}.$$
The rest of the argument follows the proof of Theorem 1.1 in \cite{XiongZheng}. If a geodesic $\gamma$ passes through $A_k$, consider the portion $\gamma|_{[a_k,b_k]}$ such that $|\gamma(a_k)|={3h(r_k)}$ and $|\gamma(b_k)|={2h(r_k)}$. Then
\begin{eqnarray*}
\int_{\gamma|_{[a_k,b_k]}} b_{\Omega}(z) |\mathrm{d}z| &=& \int_{a_k}^{b_k} b_{\Omega}(\gamma(t))|\mathrm{d}\gamma(t)| \\
&\ge& \left| \int_{a_k}^{b_k} b_{\Omega}(\gamma(t))\mathrm{d}|\gamma(t)|   \right|\\
&\gtrsim& \int^{3h(r_k)}_{2h(r_k)}\frac{1}{h(r_k) } \mathrm{d}r\\
&\gtrsim& 1.
\end{eqnarray*}
Let $n$ be the integer such that $r_{n+1}<|z|\le r_{n}$. Then any geodesic $\gamma$ starting from $z_0$ and ending at $z$ intersects a number of annuli of the form $A_k$ that is of the order of $n$. Hence
$$
d_{\Omega}(z,z_0) \gtrsim n \gtrsim \log \log \frac{1}{|z|}.
$$

$(2)$ The proof is similar to that of $(1)$ and is left to the reader.
\end{proof}

{\bf Acknowledgements.} We are grateful to Prof. Bo-Yong Chen and Dr. Yuan-pu Xiong for many inspiring discussions and critical suggestions.

\end{document}